\documentclass[10pt,a4paper]{amsart}

\usepackage{amsmath,amsfonts,amssymb,amsthm,mathrsfs}

\usepackage{color}

\usepackage[a4paper, outer=2.5cm, inner=2.5cm, heightrounded, marginparwidth=2cm, marginparsep=0.2cm, top=2cm, bottom=2cm]{geometry}
\usepackage{marginnote}

\def\qed{\hfill $\square$}

\newcommand{\R}{\mathbb{R}}
\newcommand{\N}{\mathbb{N}}
\newcommand{\un}{\mathbf{1}\!\!{\rm I}} 
\newcommand{\be}{\begin{equation}} 
\newcommand{\ee}{\end{equation}}
\newcommand{\bea}{\begin{eqnarray}} 
\newcommand{\eea}{\end{eqnarray}}
\newcommand{\bean}{\begin{eqnarray*}} 
\newcommand{\eean}{\end{eqnarray*}}
\newcommand{\rf}[1]{(\ref {#1})}

\def\dx{\,{\rm d}x}
\def\dy{\,{\rm d}y}
\def\dt{\,{\rm d}t}
\def\dr{\,{\rm d}r}
\def\dta{\,{\rm d}\tau}
\def\ds{\,{\rm d}s}
\def\dsi{\,{\rm d}\sigma}

\def\e{\varepsilon}
\def\eps{\epsilon}

\def\ei{{\rm e}_1}
\def\s{\sigma}
\def\p{\partial}

\def\g{\gamma}

\def\r{\varrho}
\def\xn{|\!|\!|}
\def\wu{\widetilde u}
\def\mn{|\!\!|}
\def\mnp{|\!\!|_{M^{d/\alpha}}}

\def\fie{\phi_\varepsilon}

\newtheorem{theorem}{Theorem}
\newtheorem{proposition}[theorem]{Proposition}
\newtheorem{lemma}[theorem]{Lemma}

\theoremstyle{definition} 
\theoremstyle{remark}
\newtheorem{remark}[theorem]{Remark}

\numberwithin{equation}{section}
\numberwithin{theorem}{section}

\author[P. Biler]{Piotr Biler}
\address[P. Biler]{Instytut Matematyczny, Uniwersytet Wroc\l awski,
 pl. Grunwaldzki 2/4, \hbox{50-384} Wroc\-\l aw, Poland}
\email{Piotr.Biler@math.uni.wroc.pl}

\author[G. Karch]{Grzegorz Karch}
\address[G. Karch]{
 Instytut Matematyczny, Uniwersytet Wroc\l awski,
 pl. Grunwaldzki 2/4, \hbox{50-384} Wroc\-\l aw, Poland}
\email{Grzegorz.Karch@math.uni.wroc.pl}

\author[J. Zienkiewicz]{Jacek Zienkiewicz}
\address[J. Zienkiewicz]{
 Instytut Matematyczny, Uniwersytet Wroc\l awski,
 pl. Grunwaldzki 2/4, \hbox{50-384} Wroc\-\l aw, Poland}
\email{Jacek.Zienkiewicz@math.uni.wroc.pl}

\title[Nonlocal model of  chemotaxis]{ Large global-in-time solutions \\ to a nonlocal model of  chemotaxis}

\begin{document}

\begin{abstract} 
We consider   the parabolic-elliptic model for the  chemotaxis with   fractional (anomalous) diffusion. Global-in-time solutions are constructed under (nearly) optimal assumptions on the size of radial initial data. 
Moreover, criteria for blowup of radial solutions in terms of suitable Morrey spaces norms are derived.   
\end{abstract}



\date{\today}

\thanks{  
The   authors were supported by the NCN grant  {%\color{red} 
2013/09/B/ST1/04412}. 
 }

\maketitle

\baselineskip=19pt
%%%%%%%%%%%%%%%%%%%%%%%%%%%%%%

%\tableofcontents
%%%%%%%%%%%%%%%%%%%%%%%%%%%%%%

\section{Introduction}
\subsection*{Formulation of the problem}
We consider in this paper the following version of the parabolic-elliptic  Keller--Segel model of chemotaxis in $d\ge 2$ space dimensions
\bea
u_t+(-\Delta)^{\alpha/2}u+\nabla\cdot(u\nabla v)&=&0,\ \ x\in {\mathbb R}^d,\ t>0,\label{equ}\\ 
\Delta v+u &=& 0,\ \   x\in {\mathbb R}^d,\ t>0,\label{eqv}
\eea
supplemented with the nonnegative initial condition 
\be
u(x,0)=u_0(x)\label{ini}.
\ee
Here the unknown variables $u=u(x,t)$ and $v=v(x,t)$ correspond to the density of the population of microorganisms ({\em e.g.} swimming bacteria or slime mold)  and the density of the chemical secreted by themselves that attracts them and makes them to aggregate. In this work, a diffusion process described by model \eqref{equ}--\eqref{ini} is given by the  fractional power of the Laplacian $(-\Delta)^{\alpha/2}$ with  $\alpha\in(0,2)$ which  is a pseudodifferential operator with a symbol $|\xi|^\alpha$, see  {\em e.g.} \cite{J1} for a comprehensive treatment of nonlocal diffusion operators. 
In case of sufficiently regular functions, we also have the following well-known representation of the fractional Laplacian with $\alpha\in(0,2)$ 
\be
-(-\Delta)^{\alpha/2}\omega(x)={\mathcal A} 
\lim_{\delta\searrow0}\int_{\{|y|>\delta\}}\frac{\omega(x-y)-\omega(x)}{|y|^{d+\alpha}}\dy,\label{fr-lap}
\ee   
where, by {\it e.g.} \cite[Th. 1]{DI}, \cite{Kw}, 
\be
{\mathcal A}={\mathcal A}({d,\alpha})=\frac{2^\alpha\Gamma\left(\frac{d+\alpha}{2}\right)}{\pi^{d/2}\left|\Gamma\left(-\frac{\alpha}{2}\right)\right|}.\label{stala-A}
\ee 

The initial datum in \eqref{ini} is a  nonnegative function $u_0\in L^1(\mathbb R^d)$ of the total mass $M=\int_{\R^d} u_0(x)\dx$ which is conserved during the evolution of (suitably regular) solutions 
\be 
M=\int_{\R^d} u(x,t)\dx \qquad \text{for all} \quad t\in [0, T).\label{M}
\ee  
Note, however, that  a natural scaling for  system \rf{equ}--\rf{eqv} 
\begin{equation}\label{scal:0}
u_\lambda(x,t)=\lambda^\alpha u(\lambda x,\lambda^\alpha t)\ \ {\rm for\ each\ \ }\lambda>0,
\end{equation}
leads to the equality $\int_{\R^d} u_\lambda(x,t) \dx = \lambda^{\alpha-d} \int_{\R^d} u(x,t) \dx$, {\it i.e.} for $\alpha\neq d$, the total   mass  of a rescaled solution $u_\lambda$ can be chosen arbitrarily  with   suitable $\lambda>0$.

\subsection*{The $8\pi$-problem in the classical case} 
Let us now describe previous results which motivated us to write this work. 
Since there is already a huge amount of literature on different models of chemotaxis, we are going to limit ourselves to those publications, which are directly related to this paper. 
We begin with the  classical case of $\alpha=2$ and $d=2$ where mass $M=8\pi$ plays a crucial role.    
Namely, if $u_0$ is a nonnegative measure of mass $M<8\pi$, then there exists a unique solution which is global-in-time, see {\em e.g.}  \cite{BM,BDP,BZ}.  
These results have been known previously for radially symmetric initial data, see  \cite{BKLN1,BKLN2,B-BCP,BKZ} for recent presentations. 
On the other hand, if $M>8\pi$, then this solution cannot be continued to a global-in-time regular one, and a finite time blowup occurs, {\em cf.} \cite{B-CM3,N1,K-O}, and \cite{BHN,B-BCP} for radially symmetric case. The radial blowup is accompanied by the concentration of mass  equal to $8\pi$ at the origin. 
In the general case, this concentration phenomenon occurs with a quantization of mass equal to $8k\pi$, $k\in\mathbb N$, {\em cf.} \cite[Ch. 15]{Suzuki}.

\subsection*{Parabolic-elliptic model in higher dimensions}
Now, we discuss the case of $\alpha=2$ and $d\geq 3$ in the model  \eqref{equ}--\eqref{ini}.   
It is well-known that problem \rf{equ}--\rf{ini} with $\alpha=2$ has a unique local-in-time mild solution $u\in {\mathcal C}([0,T); L^p(\R^d))$ 
for every $u_0\in L^p(\R^d)$ with $p>d/2$, see \cite{B-SM,K-JMAA,KS-IMUJ}. For solvability results in other functional spaces, see also \cite{BB-SM,B-W,K-JMAA,Lem,SYK}. In particular,   previous works have dealt with the existence of global-in-time solutions with small data in critical  spaces, {\em i.e.}~those which are scale-invariant under the natural scaling \rf{scal:0}, {\em cf. e.g.} \cite{B-SM,BB-SM,K-JMAA,Lem}. 

Here, as usual, a mild solution satisfies a suitable integral formulation  \rf{Duh} of the Cauchy problem  \rf{equ}--\rf{ini}.  
Due to a parabolic regularization effect,  this solution is smooth for $t>0$, hence, it satisfies the Cauchy problem in the classical sense. 
Moreover, it conserves the total mass \rf{M} and is nonnegative when $u_0\geq 0$. 
Proofs of these classical results can be found {\em e.g.} in \cite{B-SM,KS-JEE,K-O,Lem}. 

It is well known that system \rf{equ}--\rf{eqv} possesses local-in-time solutions which cannot be continued to the global-in-time ones, see \cite{BHN,N,N1,BK-JEE} for recent results.  
If $d\ge 3$,  a~sufficient condition for blowup is that  $u_0$ is well concentrated, namely
$$
\left(\frac{\int_{ \R^d} |x|^\gamma u_0(x) \dx}{\int_{ \R^d}u_0(x)\dx}\right)^{\frac{d-2}{\gamma}}\le c_dM,$$ 
for some  $0<\gamma\le 2$ and  a (small, explicit) constant $c_d>0$.
 In all these cases, at the blowup time  $0<T<\infty$, we have see $\lim_{t\nearrow T} \| u(x,t) \|_\infty=\infty$ (\cite{B-CM3,BK-JEE}). 
Results on fine asymptotics of solutions at the blowup time can be found  {\em e.g.} in  \cite{KS-SIAM,RS}. 

Criteria for a  blowup of solutions  with large concentrations can  be expressed in terms of  related critical Morrey space norms (see Remark \ref{discont} below for more details), and we have found  that the size of such a norm is also critical for the global-in-time existence versus finite time blowup. Such results  for radially symmetric (and $N$-symmetric) solutions of the  $d$-dimensional classical Keller--Segel model with $d\ge 3$ and $\alpha=2$ has been recently studied in  \cite{BKZ,BKZ-NHM}.

\subsection*{Subcritical case $\alpha\in(1,2)$}
Various results on local-in-time (and also global-in-time) solutions to the Cauchy problem  \rf{equ}--\rf{ini} with {\em subcritical}\ \  $\alpha\in (1,2)$ in various functional spaces (Lebesgue,  Besov, Morrey) have been obtained in {\em e.g.} \cite[Th. 2.2]{BW}, \cite[Th. 1.1]{B-W},  \cite[Th. 2.1]{BK-JEE}, \cite[Th. 2]{Lem} and Section \ref{smooth-sol-} below. 
They are, roughly speaking, analogous to those for $\alpha=2$.  
 Nonexistence of global-in-time solutions to problem \rf{equ}--\rf{eqv} with $\alpha<2$ corresponding to large initial conditions has been proved in \cite{BK-JEE,BKL,LR07,LR08,LRZ}.

\subsection*{Supercritical case $\alpha\in(0,1]$}
 For {\em supercritical} $\alpha\in(0,1]$ there are results on the local-in-time  solvability of \rf{equ}--\rf{ini}  with the initial data in Besov spaces  in \cite[Th. 1, Th. 2, Th. 3, Remark 10]{SYK}. 
Other solvability results   with rather smooth initial data $u_0\in H^s(\R^2)\cap L^q(\R^2)$, $s>3$, $1<q<2$, can be found  in \cite[Th. 1.1]{LRZ}, see also Theorem \ref{mainglobth} below.    
Recall that (see \cite[Remark 7]{SYK}) if $u_0\in L^1(\R^d)$ is radially symmetric and nonnegative, then the solution constructed in \cite[Th. 1, Th. 2]{SYK} is also radially symmetric, nonnegative and satisfies the $L^1$-conservation property \rf{M}.

 \subsection*{Brief description of results in this work} 
Motivated by the existence of the threshold value of mass $M=8\pi$ playing a crucial role in the study of problem \rf{equ}--\rf{ini} on the plane and with $\alpha=2$, we try to identify threshold size of initial data such that corresponding solutions of problem \rf{equ}--\rf{ini} with $d\ge 2$ and $\alpha\in(0,2)$ either exist or do not  exist for all $t\ge 0$. 
In this work we limit ourselves to nonnegative radially symmetric solutions. 
First, in Theorem \ref{sing}, we show that system \rf{equ}--\rf{eqv} has a singular stationary solution of the form $u_C(x)=s(\alpha,d)|x|^{-\alpha}$ where the constant $s(\alpha,d)$ is calculated explicitly, below. 
This singular solution plays a crucial role in our construction of   global-in-time solutions.  
In Theorem \ref{mainglobth}, we consider problem \rf{equ}--\rf{ini} with $\alpha\in(0,1)$ and we assume that a nonnegative, radial and sufficiently regular initial datum stays below the singular steady state $u_C(x)$. 
In this case, we always construct global-in-time solutions. 
Global-in-time solutions of problem \rf{equ}--\rf{ini} with $\alpha\in(1,2)$ are obtained in Theorem~\ref{2globth}. 
Here, however, we have to  assume that the initial datum stays below $u_C(x)$ in the following integral sense 
$$
\int_{\{|x|<R\}}u_0(x)\dx <\eps\int_{\{|x|<R\}}u_C(x)\dx =\eps\frac{\s_d}{d-\alpha}s(\alpha,d)R^{d-\alpha}\ \ \text{for each }\ R>0,
$$
where $\eps\in(0,1)$ is arbitrary and fixed. 
%Using methods presented in this work, we are not able to construct global-in-time solutions of problem \rf{equ}--\rf{ini} with $\alpha=1$. 

The quantity $R^{\alpha-d}\int_{\{|x|<R\}}u_0(x)\dx$ plays a crucial role in Theorem \ref{blow} where we show that some solutions  cannot  exist for all $t>0$. 
In that theorem, we show that there exists a critical constant $C_{\alpha,d}>0$  such that if $R_0^{\alpha-d}\int_{\{|x|<R_0\}}u_0(x)\dx >C_{\alpha,d}$ for some $R_0>0$, then the corresponding solution of problem \rf{equ}--\rf{ini} with $\alpha\in(0,2]$  cannot  be global-in-time. 
Theorem \ref{blow} implies also that problem \rf{equ}--\rf{ini} is locally ill-posed in the space ${\mathcal C}([0,T],M^{d/\alpha}(\R^d))$, see Remark \ref{discont} for more detail. 
At the end of this work, in Remark~\ref{r3}, we try to estimate the value  of the number $C(\alpha,d)$ and to compare it with the critical value $\frac{\s_d}{d-\alpha}s(\alpha,d)$ required in the construction of global-in-time solutions in Theorem \ref{2globth}. 

This paper is constructed in the following way.
In the next section we state and discuss all our results. 
Section \ref{sec:sing} contains calculations leading to the singular stationary solution $u_C(x)$.  
Section \ref{5} and \ref{ave} contain the proofs of Theorem \ref{g1} (for $\alpha\in(0,1)$) and Theorem \ref{g2} (for $\alpha\in(1,2)$) asserting that if an initial datum stays below the steady state then so is the corresponding solution. 
These two comparison principles allow us to construct global-in-time for $\alpha\in(0,1)$ in Theorem \ref{mainglobth}  proved in Section \ref{smooth-sol}, and for $\alpha\in(1,2)$ in Theorem \ref{2globth} proved in Section \ref{smooth-sol-}. 
Our blowup results stated in Theorem \ref{blow} are proved in Section \ref{bl}. 

The case $\alpha=2$ of classical diffusion in the Keller--Segel system is studied using different methods, and the results on the optimal conditions for global-in-time existence of radial and nonnegative solutions will appear in our forthcoming work. % \cite{BKZ2}.  

 \subsection*{Notation} 
In the sequel, $\|\cdot\|_p$ denotes the usual $L^p(\mathbb R^d)$ norm, $\|\cdot\|_{W^{s,p}}$ denotes the Sobolev space $W^{s,p}(\R^d)$ norm, and $C$'s are generic constants independent of $t$, $u$,  ...   which may, however, vary from line to line. 
The frequently used Morrey space norms are denoted by $\mn \cdot \mn_{M^p}$, for their definitions, see  \rf{hMor}.  
   Integrals with no integration limits are meant to be calculated over the whole $\R^d$. 
The relation $f(z)\sim g(z)$ as $z\to\infty$ means: $\lim_{z\to\infty} f(z)/g(z)=1$. 
%%%%%%%%%%%%%%%%%%%%%%%%

\section{Statement of results}

 As we have already mentioned in Introduction, the critical value of mass $M=8\pi$ decides whether  a nonnegative integrable initial datum in problem  \eqref{equ}--\eqref{ini} with $\alpha=2$ and $d=2$ leads to a global-in-time solution or not.
 In the case of $\alpha\neq d$,   mass cannot play such a role anymore due to the scaling \rf{scal:0}.
Thus, when studying a blowup phenomenon of solutions to problem 
\eqref{equ}--\eqref{ini}, the following natural question arises:  {\it  how to determine threshold for a size and for a singularity of an initial datum such that the corresponding solution  of problem  \rf{equ}--\rf{ini} is still regular and global-in-time?}
 In this paper, in the series of four theorems,
 we partially answer this question 
 in the case of radially symmetric nonnegative solutions
 of problem  \rf{equ}--\rf{ini} with $\alpha \in (0, 2)$.
 
We begin by emphasizing that this  question is intimately related to the existence of stationary, radial, and homogeneous    solutions of system \eqref{equ}--\eqref{eqv} which (by a scaling argument) must  take the form 
 \be
 u_C(x) \equiv   \frac{s(\alpha,d)}{|x|^\alpha} \qquad \text{for a constant}
 \quad  s(\alpha,d)>0. \label{sing}
 \ee 
For $d\ge 3$  and $\alpha=2$, the function  $u_C(x)={2(d-2)}|x|^{-2}$ is the well-known  Chandrasekhar solution of system \rf{equ}--\rf{eqv}. 
Due to its singularity at $x=0$,   it is neither weak nor distributional solution   for $d\in\{3,  4\}$.
In our first theorem, we construct  counterparts of the Chandrasekhar solutions to system   \eqref{equ}--\eqref{eqv}. 

\begin{theorem}[Singular stationary solutions]\label{Chandra}
Let $d\geq 2$, $2\alpha< d$, and 
$$s(\alpha,d)=2^\alpha\frac{\Gamma\left(\frac{d-\alpha}{2}+1\right)\Gamma(\alpha)}{\Gamma\left(\frac{d}{2}-\alpha+1\right)\Gamma\left(\frac{\alpha}{2}\right)}. 
$$ 
Then $u_C(x)=\frac{s(\alpha,d)}{|x|^{\alpha}}$ is a~distributional,  stationary solution to system \rf{equ}--\rf{eqv}.
\end{theorem}

The proof of this theorem, given  in Section \ref{sec:sing},  involves   formulas for  convolutions of Bessel potentials with explicitly given constants. 
Here, we emphasize only that the assumptions $d\ge 2$ and $2\alpha<d$ are necessary for $u_C(x)$ to be a solution in the distributions sense. 

\begin{remark}   
Note that  the limiting value of $s(\alpha,d)$ as $\alpha\to 2$ is just $s(2,d)=2(d-2)$ as is for the Chandrasekhar solution.  
 \end{remark}
 
The exact  form of stationary solutions will play a crucial role in the statements and the proofs of our next results.
In the following two  theorems, we construct    global-in-time  radially symmetric solutions to problem \rf{equ}--\rf{ini} with  $\alpha\in (0,1)$ and $\alpha\in(1,2)$, respectively,   and  with large, sufficiently regular, nonnegative  initial conditions which are below the  singular steady state $u_C$.  
Methods presented in this work cannot be applied to problem \rf{equ}--\rf{ini} with $\alpha=1$.

\begin{theorem}[Global-in-time solutions in supercritical case] \label{mainglobth}
Assume  $\alpha\in(0, 1)$, $n=2p>d+1 $ with  $p\in \mathbb N$,  $\eps \in (0,1)$ and $K>0$.  
Consider a  radially symmetric initial datum  
 $ u_0\in W^{4,n}(\R^d)\cap L^1(\R^d)\subset L^\infty(\R^d)$.  
There exists $\g_0\in(0,\alpha)$ ($\g_0$ sufficiently close to $\alpha$) and $N>0$ ($N$ sufficiently large) such that if $u_0$ satisfies 
\be
0\leq u_0(x)< \min\left\{N,\frac{K}{|x|^{\g_0}},\frac{\eps s(\alpha,d)}{|x|^\alpha}\right\}\label{zalo-0}
\qquad \text{for all} \quad x\in\R^d\setminus\{0\},
\ee then  problem  \rf{equ}--\rf{ini} has a radially symmetric, global-in-time  solution  
\be
u\in {\mathcal C}([0,\infty),W^{4,n}(\R^d))\cap {\mathcal C}^1([0,\infty),W^{3,n}(\R^d)) \label{est0}
\ee 
such that $u(t)\in  L^1(\R^d)$ for each $t>0$.
Moreover, this solution satisfies the bound 
\be
0\le u(x,t)< \min\left\{N, \frac{K}{|x|^{\g_0}}, \frac{\eps s(\alpha,d)}{|x|^\alpha}\right\}
\qquad \text{for all} \quad x\in\R^d\setminus\{0\}, \; t\geq 0.
\label{est00}
\ee
\end{theorem}

The proof of Theorem \ref{mainglobth} given in Section \ref{smooth-sol}
 is based on a {\em comparison principle involving the singular steady state} $u_C(x)={s(\alpha,d)}{|x|^{-\alpha}}$  which is rather unusual property of solutions  to models of chemotaxis. 
More precisely, we show below in Theorem \ref{g1} that if a sufficiently regular radial initial datum satisfies estimate \rf{zalo-0} then the corresponding solution must stay below a special barrier constructed with the use of  the singular steady state $u_C(x)$.

 In our next theorem, we construct global-in-time solutions in the subcritical case $\alpha\in (1,2)$ and with initial conditions in  the homogeneous Morrey    spaces $M^p(\R^d)$. 
 These spaces are defined for $1\le p<\infty$ by their norms 
 \be
 \mn u\mn_{M^p }\equiv  \sup_{R>0,\,y\in \R^d} R^{d(1/p-1)}\int_{\{|x-y|<R\}}|u|  \dx.\label{hMor}
\ee  
The key property is   another version of the comparison principle which is valid for integrated (radial) solutions.

\begin{theorem}[Global-in-time solutions in the subcritical case] \label{2globth} 
Let   $\alpha\in(1,2)$, $d>2\alpha$  and   $\eps\in (0,1)$. 
 Assume that the nonnegative radial initial datum $ u_0\in L^\infty$ satisfies 
 %$u_0\in M^p(\R^d)\cap L^\infty(\R^d)$ for some $p>d/\alpha$  and 
\be 
\int_{\{|x|<R\}}u_0(x)\dx<\min\left\{ KR^{d-\g},\eps\frac{\s_d}{d-\alpha}R^{\alpha-d}\right\} \text{ for all } R>0,\ t>0,\label{zalo-2:0} %\nonumber
\ee
for some fixed $\g\in(0,\alpha)$ and $K>0$, where the number $s(\alpha,d)$ is  defined in Theorem \ref{Chandra}. 
Then,  the corresponding solution of   system  \rf{equ}--\rf{eqv} is nonnegative, global-in-time and satisfies the estimates 
\be 
\int_{\{|x|<R\}}u(x,t)\dx<\min\left\{ KR^{d-\g},\eps\frac{\s_d}{d-\alpha}R^{\alpha-d}\right\} \text{ for all } R>0,\ t>0.\nonumber
\ee
\end{theorem}

This theorem is proved in Section \ref{smooth-sol-}. 

\begin{remark} 
Note that for the singular stationary solution $u_C(x)$ we have 
$$\int_{\{|x|<R\}}u_C(x)\dx=s(\alpha,d)\int_{\{|x|<R\}}\frac{1}{|x|^\alpha}\dx= \frac{\s_d}{d-\alpha}s(\alpha,d)R^{d-\alpha}.
$$
Thus, assumption \rf{zalo-2:0} means that the nonnegative $u_0$ is (in a certain -- averaged -- sense) below the singular steady state, analogously as in Theorem \ref{mainglobth}, assumption \rf{zalo-0}. 
\end{remark}

\begin{remark}\label{B}
Observe that if  $0\le u_0\in M^p(\R^d)$ for some $p\ge 1$ then, by definition \rf{hMor}, 
we have $\int_{\{|x|<R\}}u_0(x)\dx\le KR^{d(1-1/p)}$ for all $R>0$ and 
$K=\mn u_0\mn_{M^p}$. 
Thus,  we shall use the estimate $\int_{\{|x|<R\}}u_0(x)\dx< KR^{d-\g}$ with $\g=d/p<\alpha$ in the proof of Theorem \ref{2globth}.  
\end{remark} 

\begin{remark} \label{C}
On the other hand, if a nonnegative radial function $v=v(x)$ satisfies 
$$
\int_{\{|x|<R\}} v(x)\dx\le CR^{d-\kappa}\ \ \ {\rm for\ \ all\ \ \ } R>0
$$
with a fixed $C>0$ and $\kappa\in[1,d)$, then in fact $v$ belongs to the Morrey space $M^{d/\kappa}(\R^d)$, see Proposition \ref{prop:Mor} below for the proof. 
Thus, inequality \rf{zalo-2:0} expresses a certain assumption on $u_0$ in terms of the norm in $M^{d/\alpha}(\R^d)$. 
\end{remark}

\begin{remark}
As we have mentioned above, inequality \rf{zalo-2:0} means that the radial and nonnegative initial datum belongs to the Morrey space $M^{d/\alpha}(\R^d)$. 
This is the scaling invariant space ({\em cf.} \rf{scal:0}), and problem \rf{equ}--\rf{eqv} with $\alpha\in(1,2]$ and small initial conditions from $M^{d/\alpha}(\R^d)$ has a global-in-time solution $u\in L^\infty([0,\infty), M^{d/\alpha}(\R^d))$.
The proof of this fact for $\alpha=2$ can be found in \cite{B-SM,Lem}, however, an extension of those results to every $\alpha\in(1,2]$ is immediate. 
Theorem \ref{2globth} extends those results in the case of radial and nonnegative initial data replacing a smallness assumption in $M^{d/\alpha}(\R^d)$ by imposing inequality \rf{zalo-2:0}. 
Moreover, we show below in Theorem \ref{blow} that if $\sup_{R>0}R^{\alpha-d}\int_{\{|x|<R\}} u_0(x)\dx $ is sufficiently large, then the corresponding solution cannot be global in time. 
\end{remark}

In our last main result, we formulate  new sufficient conditions for the  nonexistence of  global-in-time solutions to problem \rf{equ}--\rf{ini}. 

\begin{theorem}[Blowup of solutions]\label{blow}
Let  $\alpha\in(0,2]$.  Consider a local-in-time, nonnegative,   classical, radially symmetric solution $u\in {\mathcal C}([0,T), L^1_{\rm loc}(\R^d))$ of   problem \rf{equ}--\rf{ini} with a nonnegative radially symmetric initial datum $u_0\in  L^1_{\rm loc}(\R^d)$.  There exists a constant $c_{\alpha,d}>0$ such that 
\begin{itemize} 
   \item[(i)] 
if 
\be
\sup_{R>0}R^{\alpha-d}\int_{\{|x|<R\}} u_0(x)\dx>c_{\alpha,d},\label{(i)}
\ee then the solution $u$ cannot exists for all $t>0$. 
   \item[(ii)]   If, moreover,  
\be
\limsup_{R\to 0}R^{\alpha-d}\int_{\{|x|<R\}} u_0(x)\dx>c_{\alpha,d},
\label{(ii)}
\ee then the solution $u(x,t)$ cannot be defined  on any time interval $[0,T]$ with some $T>0$.
\end{itemize}
\end{theorem} 
 
\begin{remark}\label{rem-blow} 
 The novelty of these blowup results consists in using local properties of solutions instead of a comparison of the total mass and moments  of a solution (like $\int |x|^\g u(x,t)\dx$) as was done in {\em e.g.} \cite{N}, \cite{KS-JEE},   \cite{BK-JEE}. 
 For   different blowup results, see  also \cite{BK-JEE} and \cite[Th. 4]{SYK}. 
\end{remark}

\begin{remark} \label{discont} 
Condition \rf{(i)} 
means that the Morrey space norm in $M^{d/\alpha}(\R^d)$ of the initial datum $u_0$ is large enough, see Proposition \ref{prop:Mor} below. 
Moreover, condition \rf{(ii)} applies only to initial conditions which are singular at the origin. 
This condition implies  that problem \rf{equ}--\rf{ini} is ill-posed in ${\mathcal C}([0,T],M^{d/\alpha}(\R^d))$ for every $T>0$.  
\end{remark}
   
\begin{remark}\label{rem:bl}
Notice that Theorem \ref{blow} holds true for $\alpha=2$, as well. 
In this case, results in Theorem \ref{blow} are generalizations and  improvements, while  their proofs are simplifications of those in  \cite{BKZ,BKZ-NHM}, where problem \rf{equ}--\rf{ini} with $\alpha=2$ was considered. 
In particular, the estimate for the number $c_{2,d}$ proved in \cite[Th. 1.1]{BKZ-NHM} was twice worse than that one in this work, {\em cf.} Remark \ref{r3} for more detail. 
\end{remark}

%\bigskip 
 %%%%%%%%%%%%%%%%%%%%%%%%%%
\section{Radial singular stationary solutions}\label{sec:sing} 

We are in a position to prove that system \eqref{equ}--\eqref{eqv} has singular radial stationary solutions.
 
\begin{proof}[Proof of Theorem \ref{Chandra}.]
Suppose that $u_C(x)$ of the form \rf{sing} satisfies time-independent system \rf{equ}--\rf{eqv} in the distributions sense.
 To determine the constant $s(\alpha,d)$ in \rf{sing} for $d\ge 2$ and some $\alpha\in(0,2)$,    observe that by equation \rf{eqv}, we have   
 $$
 \nabla v=-\frac{s(\alpha,d)}{d-\alpha}x|x|^{-\alpha}.
 $$ 
  Now, let us rewrite equation  \rf{equ} with $u=u_C$ as 
 \be
(-\Delta)^{\alpha/2}\left(|x|^{-\alpha}\right)-\frac{s(\alpha,d)}{d-\alpha} \nabla\cdot\left(x|x|^{-2\alpha}\right)=0,\label{eq-sing}
\ee
where the equality is meant in the distributions sense, and  it is valid for $2\alpha<d$. 

Now, applying the Riesz potential $ {\mathcal I}_\alpha$ 
\be
 {\mathcal I}_\alpha\omega = \frac{\Gamma\left(\frac{d-\alpha}{2}\right)}{\pi^{d/2}2^\alpha\Gamma\left(\frac{\alpha}{2}\right)}|x|^{\alpha-d}\ast\omega,\label{Riesz}
\ee
which is the inverse of $(-\Delta)^{\alpha/2}$
(see {\it e.g.} \cite[Ch. V, Sec. 1, (4)]{S}) we interpret \rf{eq-sing} as 
$$
|x|^{-\alpha}-\frac{s(\alpha,d)}{d-\alpha}  {\mathcal I}_\alpha\left((d-2\alpha)|x|^{-2\alpha}\right)=0.
$$
Recalling the formula for convolutions 
\be
|x|^{-\beta}\ast|x|^{-\gamma}=\pi^{d/2}\frac{\Gamma\left(\frac{d-\beta}{2}\right)\Gamma\left(\frac{d-\gamma}{2}\right)\Gamma\left(\frac{\beta+\gamma-d}{2}\right)}{\Gamma\left(\frac{\beta}{2}\right)\Gamma\left(\frac{\gamma}{2}\right)\Gamma\left(d-\frac{\beta+\gamma}{2}\right)}\ |x|^{d-(\beta+\gamma)},\label{convol}
\ee
valid if $0<\beta,\, \gamma<d$ and $ \beta+\gamma>d$ (see  {\em e.g.} \cite[Ch. V, Sec. 1, (8)]{S}) we may apply equation \rf{convol} to the identity 
$$
|x|^{-\alpha}-s(\alpha,d)\frac{d-2\alpha}{d-\alpha}\frac{\Gamma\left(\frac{d-\alpha}{2}\right)}{\pi^{d/2}2^\alpha\Gamma\left(\frac{\alpha}{2}\right)}|x|^{\alpha-d}\ast|x|^{-2\alpha}=0,
$$
whenever $2\alpha<d$. 
  Finally, by relation \rf{convol}, we obtain 
\bea
s(\alpha,d)&=&\frac{(d-\alpha)}{(d-2\alpha)}\frac{\pi^{d/2}2^\alpha\Gamma\left(\frac{\alpha}{2}\right)}{\Gamma\left(\frac{d-\alpha}{2}\right)} \frac{\Gamma\left(\frac{d-\alpha}{2}\right)\Gamma(\alpha)\Gamma\left(\frac{d-\alpha}{2}\right)}{\pi^{d/2} \Gamma\left(\frac{\alpha}{2}\right)\Gamma\left(\frac{d}{2}-\alpha\right) \Gamma\left(\frac{\alpha}{2}\right)}\nonumber\\
&=& 2^\alpha\frac{(d-\alpha)}{(d-2\alpha)} \frac{\Gamma\left(\frac{d-\alpha}{2}\right)\Gamma(\alpha)}{\Gamma\left(\frac{d}{2}-\alpha\right)\Gamma\left(\frac{\alpha}{2}\right)}\label{stala}\\
&=&2^\alpha\frac{\Gamma\left(\frac{d-\alpha}{2}+1\right)\Gamma(\alpha)}{\Gamma\left(\frac{d}{2}-\alpha+1\right)\Gamma\left(\frac{\alpha}{2}\right)}.
\nonumber
\eea
\end{proof}  

\begin{remark} 
It is useful to notice the following asymptotic formula
$$
s(\alpha,d) =
2^\alpha\frac{\Gamma\left(\frac{d-\alpha}{2}+1\right)\Gamma(\alpha)}{\Gamma\left(\frac{d}{2}-\alpha+1\right)\Gamma\left(\frac{\alpha}{2}\right)} 
\sim2^{\alpha/2}\frac{\Gamma(\alpha)}{\Gamma\left(\frac\alpha{2}\right)}d^{\alpha/2}\ \ \ {\rm as\ \ \ }d\to\infty
$$  
by \rf{Gamma}. 
This will be used at the end of Section \ref{bl} to an asymptotic comparison of sufficient conditions for blowup with those for global-in-time existence. 
\end{remark} 
 
 \begin{remark} 
   As a by-product of above computations, we obtain the following  useful formula valid for $2\alpha<d$ 
  \bea
 (-\Delta)^{\alpha/2}\left(|x|^{-\alpha}\right)= s(\alpha,d)\frac{d-2\alpha}{d-\alpha}|x|^{-2\alpha}
 =   2^\alpha  \frac{\Gamma\left(\frac{d-\alpha}{2}\right)\Gamma(\alpha)}{\Gamma\left(\frac{d}{2}-\alpha\right)\Gamma\left(\frac{\alpha}{2}\right)}|x|^{-2\alpha}. \label{lap-alpha}
  \eea 
Similarly as relations \rf{lap-alpha} have been derived form \rf{Riesz} and \rf{convol}, we may write for $\alpha+\gamma<d$ the following more general formula which we will use later on 
\be
(-\Delta)^{\alpha/2}\left(|x|^{-\gamma}\right) =2^\alpha\frac{\Gamma\left(\frac{d-\gamma}{2}\right)\Gamma\left(\frac{\alpha+\gamma}{2}\right)}{\Gamma\left(\frac{d-\alpha-\gamma}{2}\right)\Gamma\left(\frac{\alpha}{2}\right)}|x|^{-\alpha-\gamma}.
\label{lap-alpha-gamma}
\ee
\end{remark}

%%%%%%%%%%%%%%%%%%%%%%%%%%
\section{Pointwise comparison principle}\label{5}

In order to show that a local-in-time solution can be continued globally-in-time,
we have to deal with a problem of its {\em apriori} control.
By this reason,   we prove  two comparison principles: the pointwise comparison principle and the averaged comparison principle which roughly state that 
if a radial and regular solution begins below a singular steady state 
$u_C(x)$ given by formula \rf{sing} than it must stay below this function for all time. 

In this section, we prove a pointwise comparison for such solutions. An analogous result for radial distributions  of solutions is obtained in the next section.

\begin{theorem}[Pointwise comparison principle]\label{g1}
\par\noindent 
Let $\alpha\in(0,1)$, $d\geq 2$,  and $T>0$. 
For every $\eps\in(0,1)$ and every $K>0$ there exist $\g_0\in(0,\alpha)$ ($\g_0$ sufficiently close to $\alpha$) and $N>0$ (sufficiently large) such that  every radial solution $u\in{\mathcal C}^1(\R^d\times[0,T])$ of  system \rf{equ}--\rf{eqv} with the properties 
\be
\lim_{|x|\rightarrow\infty}|x|^{\alpha}u(x,t)=0 \qquad \mbox{uniformly on $[0,T]$},\label{decay_inf}
\ee
and %with an initial datum $u_0$  
\be
0\le u_0(x)< \min\left\{N,\frac{K}{|x|^{\gamma_0}},\frac{\eps s(\alpha,d)}{|x|^\alpha}\right\}\equiv b(x) \text{ for all } x\in\R^d ,\label{zalo-1}
\ee 
satisfies   the estimate 
\be
0\le u(x,t)< b(x)  \text{ for  all } x\in\R^d  \text{ and }\  0\le t\le T. \label{u-est1}
\ee  
\end{theorem}

First, we formulate   an elementary observation (see also \cite[Lemma 2.1]{BKZ}) which will be used in the proof of Theorem \ref{g1} as well as  in the proof of blowup in Theorem  \ref{blow} further.
 
\begin{lemma}\label{potential}
Let $u\in L^1_{\rm loc}(\R^d)$ be a radially symmetric function,   such that  $v=E_d\ast u$ with $E_2(x)=-\frac1{2\pi}\log|x|$ and $E_d(x)=\frac1{(d-2)\sigma_d}|x|^{2-d}$ for $d\ge 3$, solves the Poisson equation $\Delta v+u=0$. Here,  the area of the unit sphere ${{\mathbb S}^d\subset {\mathbb R}^d}$ is denoted by 
\be
\s_d=\frac{2\pi^{d/2}}{\Gamma\left(\frac{d}{2}\right)}.\label{sigma}
\ee 
Then 
$$
\nabla v(x)\cdot x=-\frac1{\sigma_d}|x|^{2-d}\int_{\{|y|\le |x|\}} u(y)\dy.
$$
\end{lemma}

\begin{proof}
By the Gauss--Stokes theorem,  
we obtain for the radial distribution function $M$ of $u$
\be
M(R)\equiv \int_{\{|y|\le R\}} u(y)\dy=
-\int_{\{|y|=R\}} \nabla v(y)\cdot\frac{y}{|y|}\,{\rm d}S.\label{distr-u}
\ee
Thus,  for the radial function $\nabla v(x)\cdot\frac{x}{|x|}$ and $|x|=R$,  we arrive at the identity
$$\nabla v(x)\cdot x=
\frac{1}{\sigma_d}R^{2-d}\int_{\{|y|=R\}}\nabla v(y)\cdot\frac{y}{|y|}\,{\rm d}S = -\frac1{\sigma_d}R^{2-d}M(R)$$
which completes the proof.  \end{proof} 

\begin{remark}\label{u-M}
Notice that each radial function $u=u(x)$  satisfies for $|x|=R$ the equality 
\be
u(x)=\frac{1}{\sigma_d}R^{1-d}\frac{\partial}{\partial R}M(R),\label{u-rad}
\ee 
which results immediately from the definition of $M(R)$ in \rf{distr-u} written in the polar coordinates.  
\end{remark} 
 
 \begin{lemma}\label{waru}
For each $d\in\mathbb N$, $d\ge 2$, and $\alpha\in(0,1]$, the following  inequality 
\be
{\mathcal A}\s_d\ge \alpha s(\alpha,d)\label{warunek}
\ee
holds, where $\mathcal A$ is defined in \rf{stala-A}, $s(\alpha,d)$ --- in \rf{stala} and $\s_d$ --- in \rf{sigma}. 
\end{lemma}

\begin{proof} 
We note that inequality \rf{warunek} is equivalent to  the following relation for the Gamma function 
\be
\frac{\Gamma\left(\frac{d+\alpha}{2}\right)}{\Gamma\left(1-\frac{\alpha}{2}\right)\Gamma\left(\frac{d}{2}\right)} 
\ge
\frac{\Gamma\left(\frac{d-\alpha}{2}+1\right)\Gamma(\alpha)}{\Gamma\left(\frac{d}{2}-\alpha+1\right)\Gamma\left(\frac{\alpha}{2}\right)} .\label{warunek2}
\ee 
Indeed, this is an immediate consequence of the relation 
$$
{\mathcal A}\s_d=\frac{2^\alpha \Gamma\left(\frac{d+\alpha}{2}\right)}{\pi^{d/2}\left|\Gamma\left(-\frac{\alpha}{2}\right) \right|} \frac{2\pi^{d/2}}{\Gamma\left(\frac{d}{2}\right)}, 
$$
obtained from \rf{stala-A} and \rf{sigma}, and of 
$$
\alpha s(\alpha,d)=\alpha 2^\alpha\frac{\Gamma\left(\frac{d-\alpha}{2}+1\right)\Gamma(\alpha)}{\Gamma\left(\frac{d}{2}-\alpha+1\right)\Gamma\left(\frac{\alpha}{2}\right)} 
$$
by \rf{stala}, as well as of the property of the Gamma function: 
$\frac{\alpha}{2}\left|\Gamma\left(-\frac{\alpha}{2}\right)\right|=\Gamma\left(1-\frac{\alpha}{2}\right)$.  
Now, estimate  \rf{warunek2} is, in turn, equivalent to the following one 
\be
\frac{\Gamma\left(\frac{d}{2}-\alpha+1\right)\Gamma\left(\frac{\alpha}{2}\right) \Gamma\left(\frac{\alpha}{2}\right)}{\Gamma\left(\frac{d-\alpha}{2}+1\right) \Gamma(\alpha)}\ge 
\frac{\Gamma\left(\frac{\alpha}{2}\right)\Gamma\left(1-\frac{\alpha}{2}\right) \Gamma\left(\frac{d}{2}\right)}{\Gamma\left(\frac{d+\alpha}{2}\right)},\label{warunek3}
\ee 
that is, to 
\be
{\rm B}\left(\frac{\alpha}{2},\frac{d}{2}-\alpha+1\right) {\rm B}\left(\frac{\alpha}{2},\frac{\alpha}{2}\right) 
\ge {\rm B}\left(\frac{\alpha}{2},1-\frac{\alpha}{2}\right){\rm B}\left(\frac{\alpha}{2},\frac{d}{2}\right),\label{warunek4}
\ee  
where ${\rm B}$ is the Euler Beta function defined as 
\be 
{\rm B}(\mu,\nu)=\int_0^1\tau^{\mu-1}(1-\tau)^{\nu-1}\dta =\frac{\Gamma(\mu)\Gamma(\nu)}{\Gamma(\mu+\nu)} \ \ {\rm for\ all\ \ } \mu,\, \nu>0.
\label{beta}
\ee 
 Clearly, for $\alpha=1$, inequality \rf{warunek3} is satisfied. For $\alpha\in(0,1)$,  by the H\"older inequality, we obtain 
$$
\int_0^1 \tau^{\alpha/2-1}(1-\tau)^{d/2-1}\dta \le \left(\int_0^1\tau^{\alpha/2-1}(1-\tau)^{d/2-\alpha}\dta\right)^{1/p} \left(\int_0^1\tau^{\alpha/2-1}(1-\tau)^{\alpha/2-1}\dta\right)^{1/q},
$$
and
$$
\int_0^1 \tau^{\alpha/2-1}(1-\tau)^{-\alpha/2}\dta \le \left(\int_0^1\tau^{\alpha/2-1}(1-\tau)^{d/2-\alpha}\dta\right)^{1/q} \left(\int_0^1\tau^{\alpha/2-1}(1-\tau)^{\alpha/2-1}\dta\right)^{1/p},
$$ 
with 
$$p=\frac{\frac{d}{2}-\frac{3\alpha}{2}+1}{\frac{d-\alpha}{2}}, \ \ \ \ 
q=\frac{\frac{d}{2}-\frac{3\alpha}{2}+1}{1-\alpha}, \ \ \ \ \ \ \frac1p+\frac1q=1,$$ 
since $1-\alpha> 0$.
Putting those inequalities together, we arrive at inequality \rf{warunek4}. 
\end{proof}
%\bigskip

 Let us begin the proof of the comparison principle.  
 
\noindent  
  {\it Proof of Theorem \ref{g1}.}
Let $u$ be a solution of problem \rf{equ}--\rf{ini} for an initial datum $u_0$ satisfying relations \rf{decay_inf} and   \rf{zalo-1}. 
The proof of inequality \rf{u-est1} is by contradiction. 
Suppose that there exists $t_0\in(0,T]$, 
which is  the first moment when $u(x,t)$ hits the barrier $b(x)$ defined in \rf{zalo-1}.
By {\em apriori}\, ${\mathcal C^1}$ regularity  of $u(x,t)$ and by property \rf{decay_inf}  the value of $t_0$ is well defined. 
Moreover, there exists $x_{t_0}\in\R^d$ satisfying $u(x_{t_0},t_0)=b(x_{t_0})$. 

In the following, we use the numbers
\be
R_*=\left(\frac{K}{N}\right)^{1/\g_0}\ \ \textrm{and}\ \  R_\#=\left(\frac{\eps s(\alpha,d)}{K}\right)^{1/(\alpha-\g_0)} \label{RR}
\ee
which are %correspond to 
the values of $R=|x|$ corresponding to  the intersection points of three curves  forming the graph of the barrier $b(x)$. Here, we choose $N$ so large to have $0<R_*< R_\#$. 
 We   consider an auxiliary function 
\be
\wu(x,t_0)=|x|^\g u(x,t_0), \label{wu}
\ee
where the value of $\g$ depends on $|x_{t_0}|$ in the following way 
\begin{equation}
\g=  \left\{ \begin{array}{ll}  \label{gamma}
\alpha & \textrm{if \quad $|x_{t_0}|\ge R_\#$,}\\
\gamma_0& \textrm{if \quad $R_*\le |x_{t_0}|<R_\#$,}\\ 
0 & \textrm{if \quad $|x_{t_0}|<R_*$}. 
 \end{array}\right. 
 \end{equation}  
 Here, the constant $\g_0\in(0,\alpha)$ will be chosen later on. 
It is easy to see  that  $\wu(x,t_0)$ as a function of $x$ attains its local maximum  at $x_{t_0}$. 
Indeed, 
by the choice of $\g$, the function $\wu(x,t)=|x|^\g u(x,t)$ hits the modified barrier 
$|x|^\g b(x)$ at a {\em constant} part of its graph. 
Hence, 
the existence of $x_1\ne x_{t_0}$ such that $\wu(x_1,t_0)>\wu(x_{t_0},t_0)$ would contradict the choice of $t_0$ as the first hitting point of the barrier. 
Thus, we have  
\be
\nabla \wu(x,t_0){\big|_{x=x_{t_0}}}=0.\label{gradient}
\ee 
Taking into account formula \rf{eqv}, equation  \rf{equ}   can be rewritten as 
\be
  u_t=-(-\Delta)^{\alpha/2}u+u^2-\nabla u\cdot\nabla v.\label{eqq}
 \ee 
Here, we have the identity 
$$
\nabla u=\nabla(|x|^{-\g}\wu)=-\g|x|^{-2-\g}x\wu+|x|^{-\g}\nabla\wu. 
$$ 
Thus, for radially symmetric solutions, by Lemma \ref{potential}  and formula \rf{wu} we get  
$$ 
\wu_t=-|x|^\g(-\Delta)^{\alpha/2}\left(|x|^{-\g}\wu\right) +|x|^{-\g}\wu^2 -\frac{\g}{\sigma_d}|x|^{-d}\wu M(|x|,t)-\nabla\wu\cdot\nabla v, 
$$ 
where the radial distribution function $M$ is defined in \rf{distr-u}. 
Hence, by equation \rf{gradient} we obtain at the point $x=x_{t_0}$ and $t=t_0$ 
\begin{equation}
\begin{split}
 \frac{\partial}{\partial t}\wu(x_{t_0},t){\big|_{t=t_0}} 
=&-|x_{t_0}|^\g(-\Delta)^{\alpha/2}\left(|x|^{-\g}\wu\right){\big|_{(x_{t_0},t_0)}} +|x_{t_0}|^{-\g}\wu^2(x_{t_0},t_0) \\
&-\frac{\g}{\sigma_d}|x_{t_0}|^{-d}\,\wu(x_{t_0},t_0) M(|x_{t_0}|,t_0)\equiv B(\wu)(x_{t_0}, t_0).
\end{split}
\label{pochodnaW}
\end{equation}
Our goal is to show that the right-hand side of equation \rf{pochodnaW} is strictly negative. It will give a contradiction because $\wu(x_{t_0},t)$ has to increase as a function of $t$ in a neighborhood of $t_0$ to hit the barrier $|x|^\g b(x)$
at point $x_{t_0}$ (recall that $|x|^\g b(x)$ is constant in a neighborhood of $x_{t_0}$) at a moment of time $t_0$. 

We begin by auxiliary results. For radial functions $u=u(x,t)$, abusing slightly the notation, we will simply write $\wu(x,t)=\wu(R,t)$ and  $\wu(y,t)=\wu(r,t)$,  where $R=|x|$ and $r=|y|$. 
In this new notation, we rewrite the last term on the right-hand side of \rf{pochodnaW}
as follows:
\be
\frac{\g}{\s_d}|x|^{-d}\wu(x,t)M(|x|,t)=\g R^{-d}\wu(R,t)\int_0^Rr^{d-1-\g}\wu(r,t)\dr.\label{trzeci-wyraz}
\ee 

In order to deal with  the fractional Laplacian term in \rf{pochodnaW} with  $0<\alpha<1$, we use the definition \rf{fr-lap} with the constant  \rf{stala-A}.  
Applying formula \rf{fr-lap} to the function $\omega(x)=u(x,t)=|x|^{-\g}\wu(x,t)$  we arrive at 
 \begin{align}
-(-\Delta)^{\alpha/2}\left(|x|^{-\g}\wu\right)(x,t)
&= {\mathcal A}\,P.V.\int\frac{|x|^\g\wu(x-y,t)-|x-y|^{\g}\wu(x,t)}{|x|^\g|x-y|^\g|y|^{d+\alpha}}\dy\nonumber\\
&= \wu(x,t) {\mathcal A}\,P.V.
\int\frac{1}{|y|^{d+\alpha}}\left(\frac{1}{|x-y|^\g}-\frac{1}{|x|^\g}\right)\dy \label{diss}\\ 
&\quad+\, {\mathcal A}\,P.V.\int\frac{\wu(x-y,t)-\wu(x,t)}{|x-y|^\g|y|^{d+\alpha}}\dy.\nonumber
\end{align} 
Recalling  the notation $R=|x|$,
let us express the second term on the right-hand side of \rf{diss} for radially symmetric $\wu=\wu(x,t)$ in polar coordinates as follows 
\begin{equation}\label{suma}
\begin{split}
{\mathcal A}  \lim_{\delta\searrow 0}  & \int_{\{|x-y|>\delta\}}\frac{\wu(y,t)-\wu(x,t)}{|y|^\g|x-y|^{d+\alpha}}\dy 
={\mathcal A}  \lim_{\delta\searrow 0}   \int_{\{||x|-|y||>\delta\}}\frac{\wu(y,t)-\wu(x,t)}{|y|^\g|x-y|^{d+\alpha}}\dy
\\ &={\mathcal A}\lim_{\delta\searrow 0}
\int_{\{|r-R|>\delta\}} \big(\wu(r,t)-\wu(R,t)\big)\int_{\mathbb S^d}\frac{\dsi}{|x+r\sigma|^{d+\alpha}}\, r^{d-1-\g}\dr\\
&=\lim_{\delta\searrow 0}{\mathcal A}\left(\int_0^{R-\delta}+\int_{R+\delta}^\infty\right)\big(\wu(r,t)-\wu(R,t)\big) R^{-d-\alpha}\phi\left(\frac{r}{R}\right)r^{d-1-\g}\dr, 
\end{split}
\end{equation} 
where the function $\phi$ is defined by 
 \be
 \phi(\tau)\equiv \int_{{\mathbb S}^d}\frac{\dsi}{\left|\ei+\tau\sigma\right|^{d+\alpha}}>0\label{phi}
 \ee
 with ${\mathbb S}^d$ denoting the unit sphere in $\R^d$  and $\ei=(1,0,\dots,0)\in\R^d$. 
 This function satisfies 
 $$
 \phi(0)=\sigma_d=\frac{2\pi^{d/2}}{\Gamma\left(\frac{d}{2}\right)}< \phi(\tau)\ \ {\rm for}\ \ 0<\tau< 1 \quad {\rm and}\quad\phi(\tau)\sim \s_d\tau^{-d-\alpha}\ \ {\rm as}\ \ \tau\to\infty.
 $$
Moreover, it is clear that  the function $\phi$   has a singularity at $\tau=1$.  
Now, by a direct calculation, we obtain
 \begin{equation}\label{phi:eq}
 \phi(\tau)=\sigma_d^{-1-\alpha/d}\left(1-{\tau^2}\right)^{-1-\alpha/d} \int_{{\mathbb S}^d}{\mathcal P}^{1+\alpha/d}(\tau\s,\ei)\dsi,
 \end{equation} 
 where
$$
{\mathcal P}(y,z)=\frac{1}{\s_d}\frac{1-|y|^2}{\left|z-y\right|^{d }} \text{ \ for \ } |y|<1 \text{ \ and \ }|z|=1
$$ 
is the Poisson kernel of the unit ball in $\R^d$, which for a fixed $z$ is harmonic in $y$. 
It is  classical that the function  ${\mathcal P}(y,z)^{1+\alpha/d}$ is subharmonic with respect to $y$, thus  the averages  given by formula  
$
\s_d^{-1}\int_{{\mathbb S}^d} {\mathcal P}^{1+\alpha/d}(\tau\s,\ei)\dsi
$
increase as the functions of $\tau\in(0,1)$. 
 Hence,  by equation \eqref{phi:eq}, the function $\phi(\tau)$
is a strictly increasing function on $(0,1)$.  

Next, we come back to equality \rf{pochodnaW} and we observe that its right-hand side can be written 
by \rf{trzeci-wyraz}, \rf{diss} and \rf{suma} in the following way 
\begin{equation}\label{Bwu}
\begin{split}
B(\wu)(x_{t_0}, t_0) =&
 \wu(x,t)\,{\mathcal A}\;P.V.
\int\frac{1}{|y|^{d+\alpha}}\left(\frac{1}{|x-y|^\g}-\frac{1}{|x|^\g}\right)\dy\\
&+
\lim_{\delta\searrow 0} B_\delta (\wu)(x_{t_0}, t_0),
\end{split}
\end{equation}
where 
\begin{equation}\label{Bdelta}
\begin{split}
B_\delta (\wu)(x_{t_0},& t_0) \\
=&
{\mathcal A}\left(\int_0^{R_0-\delta}+\int_{R_0+\delta}^\infty\right)\big(\wu(r,t_0)-\wu(R_0,t_0)\big) R_0^{-d-\alpha}\phi\left(\frac{r}{R_0}\right)r^{d-1-\g}\dr\\
&\qquad\qquad -\g\wu(R_0,t_0) \int_0^{R_0-\delta}r^{d-1-\g}\wu(r,t_0)\dr\\
=&\int_0^{R_0-\delta}\wu(r,t_0) r^{d-1-\g}\left(
R_0^{\g-\alpha}{\mathcal A}\phi\left(\frac{r}{R_0}\right)-\g\wu(R_0,t_0)
 \right)\dr\\
 &\qquad \qquad+\int_{R_0+\delta}^\infty\big(\wu(r,t_0)-\wu(R_0,t_0)\big) R_0^{-d-\alpha}\phi\left(\frac{r}{R_0}\right)r^{d-1-\g}\dr.
\end{split}
\end{equation}
We are in a position to show that the right-hand side of equality \eqref{pochodnaW} is strictly negative by finding the maximizer of $B_\delta (\wu)(x_{t_0}, t_0)$ for each fixed $\delta>0$
on the set of all nonnegative functions satisfying 
\be
\wu(x,t)
=|x|^\g u(x,t), \quad \text{where} \quad u(x,t)\leq  
b(x)=
\min\left\{N ,\frac{K}{|x|^{\g_0}},\frac{\eps s(\alpha,d)}{|x|^{\alpha}}
\right\}, \label{constr}
\ee 
and where the parameter $\gamma$ is defined in \rf{gamma}.

We consider three separate cases depending on the value $R_0=|x_{t_0}|$. 
%\medskip 

\noindent  {\bf Case 1.}
$R_0=|x_{t_0}|\ge R_\#=\left({\eps s(\alpha,d)}/{K}\right)^{1/(\alpha-\g)}$.  Here, by definition \eqref{gamma}, 
we have $\g=\alpha$ and by the definition of $|x_{t_0}|$, we obtain $\wu(x_{t_0},t_0)=\epsilon s(\alpha,d)$.
We extend  the class of considered functions $\wu$ by looking  for the maximum value of quantity $B_\delta (\wu)(x_{t_0}, t_0)$ in \rf{Bdelta} for the functions satisfying $\wu(x,t)=|x|^\alpha u(x,t)\le \eps s(\alpha,d)$.    
In this case, the first term on the right-hand side of \rf{Bdelta} reduces to 
\begin{equation}\label{c1:0}
\int_0^{R-\delta} \wu(r,t_0)
r^{d-1-\alpha}
\left({\mathcal A}\phi\left(\frac{r}{R}\right)-\alpha \epsilon s(\alpha,d) \right)\dr.
\end{equation}
Since 
$\s_d=\phi(0)<\phi(\tau)$ for $\tau>0$, $\epsilon \in (0,1)$,  and ${\mathcal A}\s_d\ge \alpha s(\alpha,d)$ for each $\alpha\in (0,1)$ (see Lemma \ref{waru}), we obtain 
\be
{\mathcal A}\phi\left(\frac{r}{R}\right)-\alpha \epsilon s(\alpha,d) 
\geq  {\mathcal A}\s_d- \alpha \epsilon s(\alpha,d)>0.\label{star}
\ee
Hence, the integral in \rf{c1:0} increases as $\wu(r,t_0)$ increases with respect to $r$ and, under the constraint $\wu(x,t)\le \eps s(\alpha,d)$,
 its maximum  is attained at a constant function $\wu(r,t)=\wu(R_0,t_0)=\epsilon  s(\alpha,d)$.

The integrand of the second integral on the right-hand side of \rf{Bdelta} is nonpositive also  because of the constraint   
$\wu(r,t_0)\leq \epsilon s(\alpha, d)=\wu(R_0,t_0)$ by the definition of $R_0$ and $t_0$.   
Its maximum equals zero  and this is attained at the constant function $\wu(r,t_0)\equiv \epsilon  s(\alpha,d)$, as well.

Consequently, for each $\delta>0$, the quantity $B_\delta (\wu)(x_{t_0}, t_0)$  attains its maximum  at the constant function $\wu(r,t_0)=\epsilon  s(\alpha,d)$.
Now, we may pass   to the limit  $\delta\searrow 0$ using the formula from \rf{Bwu} to conclude that 
 the right-hand side of  equality  \rf{pochodnaW} attains its maximum (under the constraint  $\wu(x,t)\le \eps s(\alpha,d)$)
 at the constant function $\wu(r,t)\equiv\epsilon  s(\alpha,d)$. Hence, using  formulas \rf{trzeci-wyraz}  and \rf{lap-alpha}, we have got 
\begin{equation}\nonumber 
\begin{split}
\frac{\partial}{\partial t}&\wu(x_{t_0},t){\big|_{t=t_0}} \\
 \le&-R_0^\alpha(-\Delta)^{\alpha/2}\left(|x|^{-\alpha}\epsilon  s(\alpha,d)\right){\big|_{|x|=R_0}} +R_0^{-\alpha}(\epsilon  s(\alpha,d))^2 
 -\frac{\alpha }{d-\alpha}R_0^{-\alpha}( \epsilon  s(\alpha,d))^2   \\
=&s(\alpha,d)^2 
R_0^{-\alpha}\left(-\epsilon \frac{d-2\alpha}{d-\alpha}+\epsilon^2-\epsilon^2 \frac{\alpha }{d-\alpha}\right)\\
=&s(\alpha,d)^2 
R_0^{-\alpha}\left( \frac{d-2\alpha}{d-\alpha}\right)\epsilon (-1+\epsilon)
<0,
\end{split}
\end{equation}
where the last inequality is obtained because $\alpha\in (0,1)$ and $\epsilon \in (0,1)$. 
%\medskip

\noindent{\bf Case 2.} 
$R_*\le R<R_\#$, where $R_*$ and $R_\#$ are defined in \rf{RR}.  
Similarly as was in Case 1, we look for the maximum of $B_\delta$ in \rf{Bdelta} within an extended class of admissible functions $\wu(x,t)\le K$.  
Let us first show that ${\mathcal A}\s_d\ge \g_0 KR_0^{\alpha-\g_0}$ where $\g_0\in(0,\alpha)$ is arbitrary at this stage of the proof. 
Indeed, using inequalities \rf{star} and $R_0\le R_\#$, we get 
$$
{\mathcal A}\s_d\ge \alpha s(\alpha,d)>\g_0\eps s(\alpha,d) =\g_0 KR_\#^{\alpha-\g_0}\ge \g_0 KR_0^{\alpha-\g_0}. $$
Hence, again as before, the first term on the right-hand side of \rf{Bdelta} is nonnegative in the class of functions satisfying $0\le \wu(x,t_0)\le K$. 
Thus, as in Case 1, passing to the limit $\delta\searrow 0$, and under the constraint $0\le \wu(x,t_0)\le K$, we obtain that the constant function $\wu(x,t_0)\equiv K$ maximizes the right-hand side of \rf{Bwu}, {\em i.e.} 
$$
\frac{\partial}{\partial t}\wu(x_{t_0},t){\big|_{t=t_0}} \le -R_0^\g(-\Delta)^{\alpha/2}\left(|x|^{-\g}\right){\big|_{|x|=R_0}}+R_0^{-\g}K^2-\frac{\g K^2}{d-\g}R_0^{-\g}.$$ 
To continue, we recall that 
$$
R^\g(-\Delta)^{\alpha/2}(|x|^{-\g})(R)=R^{-\alpha}C_{\alpha,\g}
$$
with $C_{\alpha,\g} =\frac{2^\alpha\Gamma\left(\frac{\alpha+\gamma}{2}\right)\Gamma\left(\frac{d-\gamma}{2}\right)}{\Gamma\left(\frac{d-\alpha-\gamma}{2}\right)\Gamma\left(\frac{\alpha}{2}\right)}$ --- 
a consequence of formula \rf{lap-alpha-gamma}. 
We also need the inequality $KR_0^{\alpha-\g_0}\le \eps s(\alpha,d)$,  which is obvious by the definition of $R_\#$.  
Note also that we have got the inequality 
$$
-C_{\alpha,\alpha}+s(\alpha,d)\left(1-\frac{\alpha}{d-\alpha}\right)\le 0, $$
because this is equivalent to estimate \rf{warunek}. 
Thus, given $\eps\in(0,1)$ there exists $\g_0\in(0,\alpha)$, sufficiently close to $\alpha$, such that 
\bea
\frac{\partial}{\partial t}\wu(x_{t_0},t){\big|_{t=t_0}}
&\le& \frac{K}{R^\alpha}\left(-C_{\alpha,\g}+KR^{\alpha-\g}\left(1-\frac{\g}{d-\g}\right)\right)\nonumber\\
&\le& \frac{K}{R^\alpha}\left(-C_{\alpha,\g}+\epsilon s(\alpha,d )\left(1-\frac{\g}{d-\g}\right)\right)<0.\nonumber
\eea 
%\medskip

\noindent{\bf Case 3.} 
$R_0=|x_{t_0}|\le R_*=\left(K/N\right)^{1/\g_0}$. 
Here, by definition \rf{gamma}, we have $\g=0$. 
Thus, equation \rf{pochodnaW} reduces to  
\be
\frac{\partial}{\partial t}\wu(x_{t_0},t){\big|_{t=t_0}}
=-(-\Delta)^{\alpha/2}\wu+\wu^2. \label{poW}
\ee 
We estimate the right-hand side of equation \rf{poW} under the constraint 
\be 
0\le \wu(x,t)\le \min\left\{N,\frac{K}{|x|^{\g_0}}\right\} \equiv \omega(x).\label{u-comp}
\ee  
By the definition of $x_{t_0}$ we have $u(x_{t_0},t_0)=N$, and thus
\be
-(-\Delta)^{\alpha/2}u(x_{t_0},t_0)+u^2(x_{t_0},t_0)
={\mathcal A}\int\frac{u(x,t_0)-u(x_{t_0},t_0)}{|x-x_{t_0}|^{d+\alpha}}\dx+N^2.\label{wyr1}
\ee
Since $\wu(x,t_0)\le \omega(x)$ and $\wu(x_{t_0},t_0)=N=\omega(x_{t_0})$, we have 
\begin{equation}\label{wyr2}
\begin{split}
-(-\Delta)^{\alpha/2}u(x_{t_0},t_0)
&\le -(-\Delta)^{\alpha/2}\omega(x_{t_0})\\ 
&={\mathcal A}\int_{R_0}^\infty(\omega(r)-N)\int_{{\mathbb S}^d}\frac{\dsi}{|r\s-x_{t_0}|^{d+\alpha}}\dr\\
&\le{\mathcal A} \int_{R_0}^\infty (u(r,t_0)-N)r^{-d-\alpha} \int_{{\mathbb S}^d}\frac{\dsi}{\left|\frac{x_{t_0}}{r}-\s\right|^{d+\alpha}}\dr.
\end{split}
\end{equation} 
Remember that $\omega(r)-N=0$ for $r\le R_0$ and $\omega(r)-N<0$ for $r>R_0$. 
Moreover, as in the case of the function $\phi$ in \rf{phi}, the following quantity 
$$\int_{{\mathbb S}^d}\frac{\dsi}{|r\s-x_{t_0}|^{d+\alpha}}$$ is increasing as a function of $|x_{t_0}|$. 
Therefore, the maximal value of $-(-\Delta)^{\alpha/2} \omega(x_{t_0}) +N^2$ is attained at $x_{t_0}=0$. 
Now, we come back to equation \rf{wyr1}. 
Using the above estimates and the relation $NR_*^{\g_0}=K$, we obtain 
\begin{equation}\label{wyr3}
\begin{split}
\frac{\partial}{\partial t}\wu(x_{t_0},t){\big|_{t=t_0}}
&\le -{\mathcal A}\int_{\{|x|>R_*\}}\left(N-\frac{K}{|x|^{\g_0}}\right)\frac{\dx}{|x|^{d+\alpha}}+\left(\frac{K}{R_*^{\g_0}}\right)^2\\ 
&={\mathcal A}\s_d\int_{R_*}^\infty \left(\frac{K}{R_*^{\g_0}}-N\right)\frac{r^{d-1}\dr}{r^{d+\alpha}}+\left(\frac{K}{R_*^{\g_0}}\right)^2\\
&={\mathcal A}\s_d\left(\frac{1}{\alpha+\g_0}\frac{K}{ R_*^{\alpha+\g_0}}-\frac{K}{\alpha R_*^\alpha}\right)+    \left(\frac{K}{R_*^{\g_0}}\right)^2\\ 
&=\frac{K}{R_*^{\alpha+\g_0}}\left(\left(\frac{1}{\alpha+\g_0}-\frac{1}{\alpha}\right) {\mathcal A}\s_d +KR_*^{\alpha-\g_0}\right).
\end{split}
\end{equation}  
Since $\frac{1}{\alpha+\g_0}-\frac{1}{\alpha}<0$, we may choose $N$ sufficiently large so that  $KR_*^{\alpha-\g_0}=K (K/N)^{\alpha/\g_0-1}$ is sufficiently small so that $\frac{\partial}{\partial t}\wu(x_{t_0},t){\big|_{t=t_0}}<0$ by \rf{wyr3}.   
This completes the proof of Theorem \ref{g1}. 
\qed

%%%%%%%%%%%%%%%%%%%%%%%%%%
 \section{Averaged comparison principle}\label{ave}

We prove in this section a counterpart of Theorem \ref{g1} for radial distributions  of solutions.  
\begin{theorem}\label{g2}
\par\noindent
Let $d\ge 3$  and $\alpha\in(1,2)$ be such that $2\alpha<d$.
Consider a solution $u\in{\mathcal C}^2(\R^d\times[0,T])$  of system \rf{equ}--\rf{eqv} with the radially symmetric  initial data $u_0\ge 0$ satisfying the integrated bound   
\be
 M(R,0)= \int_{\{|x|<R\}} u_0(x)\dx< \min\left\{KR^{d-\gamma},\eps sR^{d-\alpha}\right\}\equiv b_I(R)\ \ {\rm for\ all\ \ }R>0,\label{zalo-2}
\ee
for some  $\eps\in(0,1)$, $\g\in(0,\alpha)$, $K\in (0, \eps s) $,  and $s = \frac{\s_d }{d-\alpha}s(\alpha,d)$  (the number $s(\alpha,d)$ is  defined in \rf{stala}). 
Then there exists   $\g_0=\g_0(\alpha,\eps)\in(0,\alpha)$ such that for each initial  condition \rf{ini} satisfying condition \rf{zalo-2} with  a certain $\g\in(\g_0,\alpha)$  the inequality 
\be
0\le M(R,t)=\int_{\{|x|<R\}}u(x,t)\dx < b_I(R) 
\label{est2}
\ee 
is satisfied for all $R>0$ and $t\in[0,T]$. 
\end{theorem} 

First, we need the following asymptotic result in the proof of this   comparison principle. 
\begin{lemma}\label{indicator} 
Let $\alpha\in(0,2)$. 
The fractional Laplacian of the indicator function of the unit ball $j(r)\equiv(-\Delta)^{\alpha/2}\un_{B_1}(x)$    satisfies the relation 
\be
j(r)=C_{\alpha,d}\,{\rm sgn}(1-r)(1-r)^{-\alpha}+
 \left\{ \begin{array}{ll} 
{\mathcal O}(1) & \textrm{if \quad $0<\alpha<1$,}\\
{\mathcal O}\left(\left|\log r\right|\right)& \textrm{if \quad $\alpha=1$,}\\ 
{\mathcal O}\left(r^{1-\alpha}\right) & \textrm{if \quad $1<\alpha<2$},
 \end{array}\right. \  \ {\rm as\ \ } r=|x|\to 1,\label{distrj}
\ee 
with some number $C_{\alpha,d}>0$. 
Moreover,  $j=j(r)$ is an increasing function on $(0,1)\cup(1,\infty)$,   $j(r)>0$ for $r\in(0,1)$, and $j(r)<0$ if $r\in(1,\infty)$. 
\end{lemma}

\begin{proof} 
Observe that from definition \rf{fr-lap}
$$
j(r)=j(|x|)={\mathcal A}\int_{\{|y|<1\}}\frac{\un_{B_1}(x)-\un_{B_1}(x-y)}{|y|^{d+\alpha}}\dy.
$$ 
Thus, 
the statements on the sign of $j=j(r)$ for $r=|x|<1$ and $r=|x|>1$   as well as those on the monotonicity of $j$ are clear; it suffices to check  the value of $j(|x|)$ when $\un_{B_1}(x-y)\neq 0$. 

Now, without loss of generality we may consider $x=(r,0,\dots,0)$ with $r>0$, since the problem is rotationally invariant. 
First, we show that for the half-space 
$ \Pi=(-\infty,1]\times\mathbb R^{d-1}$,  we have 
$$
(-\Delta)^{\alpha/2}\un_\Pi(x)=C_{\alpha,d}\,{\rm sgn}(1-r)|1-r|^{-\alpha}.
$$  
Indeed, if $x\in\Pi$, then denoting $y=(y_1,\bar y)$ with $\bar y\in \mathbb R^{d-1}$, we have 
\begin{equation*}
\begin{split}
(-\Delta)^{\alpha/2}\un_\Pi(x) &= {\mathcal A}\int_{\{x-y\notin\Pi\}}\frac{\dy}{|y|^{d+\alpha}}\nonumber\\ 
&={\mathcal A}\int_{-\infty}^{r-1}{\rm d}y_1\, \int_{\mathbb R^{d-1}}\frac{{\rm d}\bar y}{(y_1^2+|\bar y|^2)^{(d+\alpha)/2}}\nonumber\\
 &={\mathcal A}\s_{d-1}\int_{-\infty}^{r-1}|y_1|^{-1-\alpha}{\rm d}y_1\,  \int_0^\infty\frac{\r^{d-2}\,{\rm d}\r}{(1+\r^2)^{(d+\alpha)/2}}\nonumber\\
 &=C_{\alpha,d}|1-r|^{-\alpha}. 
\end{split}
\end{equation*} 
Similarly as above, for $x\notin\Pi$ we have 
$$
(-\Delta)^{\alpha/2}\un_\Pi(x)=-{\mathcal A}\int_{\{x-y\in\Pi\}}\frac{\dy}{|y|^{d+\alpha}}=-C_{\alpha,d}|1-r|^{-\alpha}$$ 
again with 
$$
C_{\alpha,d}={\mathcal A}\s_{d-1}\alpha^{-1}\int_0^\infty \frac{\r^{d-2}\,{\rm d}\r}{(1+\r^2)^{(d+\alpha)/2}}.
$$ 

To complete the proof, it suffices to show that for $P=\Pi\setminus B_1$ with  the unit ball  $B_1\subset \mathbb R^d$ centered at the origin, the estimate 
$$
\big|(-\Delta)^{\alpha/2}\un_P(x)\big|\le \left\{ \begin{array}{ll} 
{\mathcal O}(1) & \textrm{if \quad $0<\alpha<1$,}\\
{\mathcal O}\left(\left|\log r
\right|\right)& \textrm{if \quad $\alpha=1$,}\\ 
{\mathcal O}\left(r^{1-\alpha}\right) & \textrm{if \quad $1<\alpha<2$}.
 \end{array}\right.  
$$
holds as $r\to 1$.  
By the translational invariance and homogeneity, it suffices to consider the annulus $A_\r$ centered at $x$: $A_\r=\{z:\r<|x-z|<2\r\}.$ 
Clearly, $P\cap A_\r =\emptyset$ holds whenever $\r<\frac12|1-r|$ and the volume of $P\cap A_\r$ is less than $C\r^{d-1}\r^2=C\r^{d+1}$.  
Therefore, splitting the integration domain into dyadic pieces we get 
\begin{align}
\int_P\frac{\dy}{|x-y|^{d+\alpha}} &\le \int_{\{ y:\, |x-y|\ge 1\}}\frac{\dy}{|x-y|^{d+\alpha}} +\sum_{\left\{\frac12|1-r|\le\r<1,\, {\rm dyadic}\right\}}\int_{P\cap A_\r}\frac{\dy}{|x-y|^{d+\alpha}}\nonumber\\
&\le C+C\sum_{\left\{\frac12|1-r|< \r\ {\rm dyadic}\right\}}\frac{\r^{d+1}}{\r^{d+\alpha}}\nonumber\\
&\le   \left\{ \begin{array}{ll} 
c & \textrm{if \quad $0<\alpha<1$,}\\
c\left|\log r\right|& \textrm{if \quad $\alpha=1$,}\\ 
cr^{1-\alpha} & \textrm{if \quad $1<\alpha<2$}.
 \end{array}\right. 
\end{align} 
\end{proof} 
  
\noindent {\em Proof of Theorem \ref{g2}.} 
Begin with an arbitrary $K>0$ and $\g\in(0,\alpha)$ which will be specified later on. 
Similarly as in the proof of Theorem \ref{g1}, we proceed by contradiction and we define $t_0$ as the first moment when $M(R,t)$ hits the barrier $b_I(R)$. 
Because of the inequalities  $M(R,t)\le M$ and inequality in \rf{zalo-2}, the number  $t_0$ is well defined; moreover,  for some $x_{t_0}\in\R^d$ such that  $R_{t_0}=|x_{t_0}|$, we have $M(R_{t_0},t_0)=b_I(R_{t_0})$. 
  
  Define $R_\#=\left(\eps s/K\right)^{1/(\alpha-\g)}>0 \ $ 
  as the intersection point of the graphs $ KR^{d-\gamma}$ and $\eps sR^{d-\alpha}$ which exists because $\g<\alpha$. We  consider the function
\be
0\le z(r,t)=r^{\g-d}M(r,t), \label{wuz}
\ee
where $u$ solves problem \rf{equ}--\rf{ini} and $M(r,t)$ is the radial  distribution function of $u=u(x,t)$ defined  in \rf{distr-u}. 
Moreover, we replace in formula \rf{wuz} the exponent $\g$ by $\alpha$ if $|x_{t_0}|\ge R_\#$. 

Let us now derive equalities needed in the remainder of this proof. 
Similarly as was in the proof of Theorem \ref{g1}, one can prove that $z(r,t)$ as a function of $r$ attains  its local maximum at $R_{t_0}$. Hence
\be
\frac{\partial}{\partial r} z(R_{t_0},t_0)=0 \ \ \ {\rm and}\ \ \ \frac{\partial^2}{\partial r^2} z(R_{t_0},t_0)\le 0.\label{gradient-}
\ee
By relation \rf{u-rad} we may express $M_r$ in terms of $z=z(.,t)$ in the following way 
\be
M_r=\frac{\partial}{\partial r}(r^{d-\g}z)=(d-\g)r^{d-\g-1}z+r^{d-\g}z_r.\label{der}
\ee 
 Using the definition  of the fractional Laplacian   $(-\Delta)^{\alpha/2}$ and elementary computations, we obtain 
\be 
 \int_{\{|y|<R\}}(-\Delta)^{\alpha/2}u(y,t)\dy 
=\int J_R(x) u(x,t)\dx,
 \label{lapl} 
\ee   
where 
$$
J_R(x)\equiv (-\Delta)^{\alpha/2} \un_{B_R}(x)=R^{-\alpha}j\left(\frac{|x|}{R}\right)\ \ \ {\rm with}\ \ \ j(|x|)=(-\Delta)^{\alpha/2}\un_{B_1}(x).
$$ 
Properties of the function $j(r)$ have been studied above in Lemma \ref{indicator}. 

 For radially symmetric solutions $u(y,t)=u(r,t)$, and $r=|y|$, we have by the Gauss--Stokes theorem, Lemma \ref{potential}, equations  \rf{u-rad} and \rf{der}  
\begin{equation}\label{nonlin}
\begin{split}
-\int_{\{|y|<R\}}\nabla \cdot(u \nabla v) \dy 
-&\int_{\{|y|=R\}}\frac{y}{|y|}\cdot (u\nabla v)\,{\rm d}S\!\!\!\\ 
=&\frac{1}{\sigma_d^2}\int_{\{|x|=R\}} \frac{1}{R}R^{2-2d}M(R,t)M_r(R,t) \,{\rm d}S\\
=&\frac{1}{\sigma_d}R^{1-d}M(R,t)M_r(R,t)\\
=&
 \frac{d-\g}{\sigma_d}R^{d-2\g}z(R,t)^2\\
& +\frac{1}{\sigma_d}R^{d-2\g+1}z(R,t)z_r(R,t).
\end{split}
\end{equation}
Thus, using equations  \rf{gradient-},  \rf{lapl} and \rf{nonlin},  we obtain 
\begin{equation}\label{pochodna}
\begin{split}
\frac{\partial}{\partial t}&z(R_{t_0,}t){\big|_{t=t_0}}\\
=&
R_{t_0}^{\g-d}\frac{\rm d}{\dt}{\big|_{t=t_0}}\int_{\{|y|<R_{t_0}\}}u(y,t)\dy\\ 
 =& R_{t_0}^{\g-d}\left(-\int_{\{|y|<R_{t_0}\}}(-\Delta)^{\alpha/2}u(y,t_0)\dy 
+\frac{d-\g}{\sigma_d}R_{t_0}^{d-2\g}z^2\right) \\ 
=& \lim_{\delta\searrow 0} R_{t_0}^{\g-d}\left(\int_{\{|y|<R_{t_0}-\delta\}}   
u(y,t_0)\left(-R_{t_0}^{-\alpha}j\left(\frac{|y|}{R_{t_0}}\right)  
+R_{t_0}^{-\g}\frac{d-\g}{\sigma_d}z(R_{t_0},t_0)\right)\dy   \right.\\
 &+ 
 \left.
 \int_{\{|y|>R_{t_0}+\delta\}}u(y,t_0)\left(-R_{t_0}^{-\alpha}j\left(\frac{|y|}{R_{t_0}}\right)\right)\dy\right)\\ 
 =&\lim_{\delta\searrow 0} R_{t_0}^{\g-\alpha}\left(\int_{\{|y|<1-\delta\}}u(R_{t_0}y,t_0)\left(-j(|y|)+R_{t_0}^{\alpha-\g}\frac{d-\g}{\sigma_d}z(R_{t_0},t_0)\right)\dy\right. \\
 &- \left.\int_{\{|y|>1+\delta\}}u(R_{t_0}y,t_0)j(|y|)\dy\right). 
\end{split}
\end{equation}
Our ultimate goal is to prove that  
$\frac{\partial}{\partial t}z(R_{t_0},t){\big|_{t=t_0}}<0$ 
 which  implies  that $z(r,t)$ decreases in time  when it hits the barrier $b_I(r)$ at the point $(R_{t_0},t_0)$, and  which is in a contradiction with the fact that $z(r,t)$ attains a local maximum at this point.   
Recalling that $KR_\#^{\alpha-\g}=\eps s$ we consider two cases.

\noindent  {\bf Case 1.} 
We assume that $R_{t_0}\le R_\#$. 
Hence, by the definition of $R_{t_0}$, we obtain that $z(R_{t_0},t_0)=KR_{t_0}^{d-\g}$. 
We will find  the upper bound of the  right-hand side of formula \rf{pochodna} under the pointwise constraint
$$
M(r,t)=
\s_d\int_0^r u(\rho,t)\rho^{d-1}\,{\rm d}\rho\le \frac{\s_d}{d-\g}Kr^{d-\g}\ \ {\rm for\ all\ \ } 0<r<\infty. 
$$ 

\noindent  {\bf Case 2. }
 Now we suppose that $R_{t_0}\ge R_\#$, \   hence $z(R_{t_0},t_0)=\epsilon s$.
We will find the upper bound of the  right-hand side of formula \rf{pochodna} with $\alpha=\g$ under the pointwise constraint
$$
M(r,t)=\s_d\int_0^r u(\rho,t)\rho^{d-1}\,{\rm d}\rho\le \epsilon\frac{\s_d}{d-\alpha}s(\alpha,d)r^{d-\alpha}\ \ {\rm for\ all\ \ } 0<r<\infty. 
$$ 

We deal with both cases simultaneously with the goal to obtain $\frac{\partial}{\partial t}z(R_{t_0},t){\big|_{t=t_0}}<0$. 
We fix $\g\le \alpha$ and, under the constraints either of Case 1 or of Case 2, we  compute the upper bound of the expression 
\begin{equation}\label{I1}
\begin{split}
   \int_0^{R_{t_0}-\delta}\s_d&u(r,t_0)r^{d-1}\left(R_{t_0}^{\alpha-\g}\frac{d-\g}{\sigma_d}z(R_{t_0},t_0)-j(r)\right)\dr \\
 =&\int_0^{R_{t_0}-\delta}M_r(r,t_0)\left(sR_{t_0}^{\alpha-\g}\frac{d-\g}{\sigma_d}z(R_{t_0},t_0)-j(r)\right)\dr  \\
 =& \bigg( M(r,t_0)\left(R_{t_0}^{\alpha-\g}\frac{d-\g}{\sigma_d}z(R_{t_0},t_0)-j(r)\right)\bigg|_0^{R_{t_0}-\delta} \\
&+ \int_0^{R_{t_0}-\delta}(M(r,t_0)r^{\g-d})r^{d-\g}j'(r)\dr \bigg) \\
 \le& \bigg( M(r,t_0)\left(R_{t_0}^{\alpha-\g}\frac{d-\g}{\sigma_d}z(R_{t_0},t_0)-j(r)\right)\bigg|_0^{R_{t_0}-\delta}\\
&+\int_0^{R_{t_0}-\delta}z(R_{t_0},t_0)j'(r)r^{d-\g}\dr \bigg) \\
\le & \s_d s(\alpha,d)\left(\frac{s(\alpha,d)}{d-\alpha} - (d-\g)z(R_{t_0},t_0)\int_0^{R_{t_0}-\delta}r^{d-\g-1}j(r)\dr\right)\\
&+ j(R_{t_0}-\delta)(-M(R_{t_0}-\delta,t_0)+M(R_{t_0},t_0)).
\end{split}
\end{equation}  
Similarly, we have the upper bound for the integral over large $|y|$
 \begin{equation}\label{I2} 
 \begin{split}
 -\int_{R_{t_0}+\delta}^\infty M_r(r,t_0)&j(r)\dr \\
 =&  \bigg(-M(r,t_0)j(r)\bigg|_{R_{t_0}+\delta}^\infty +\int_{R_{t_0}+\delta}^\infty (M(r,t_0)r^{\g-d})r^{d-\g}j'(r)\dr\bigg)  \\
 \le &  \bigg(-M(r,t_0)j(r)\bigg|_{R_{t_0}+\delta}^\infty +\int_{R_{t_0}+\delta}^\infty z(R_{t_0},t_0)r^{d-\g}j'(r)\dr \bigg) \\
=&- (d-\g)z(R_{t_0},t_0)\int_{R_{t_0}+\delta}^\infty r^{d-\g-1}j(r)\dr\\ 
&+j(R_{t_0}+\delta)(M(R_{t_0}+\delta,t_0)-M(R_{t_0},t_0)).  
 \end{split}
 \end{equation}
Since $u\in {\mathcal C}^2, R_{t_0}>0$, we have by the asymptotic formula \rf{distrj} that 
\begin{equation*}
\begin{split} 
\lim_{\delta\searrow 0}\Big(&j(R_{t_0}-\delta)\big(-M(R_{t_0}-\delta,t_0)+M(R_{t_0},t_0)\big) \\
&+j(R_{t_0}+\delta)\big(M(R_{t_0}+\delta,t_0)-M(R_{t_0},t_0)\big)\Big)=0.
\end{split}
\end{equation*}

By the above computations, 
 the upper bound for the derivative $\frac{\partial}{\partial t}z(x_{t_0},t)\big|_{t=t_0}$  is obtained by evaluating \rf{pochodna} for $u(x,t)=K|x|^{-\g}$  in Case 1, and  for $u(x,t)=\epsilon s(\alpha,d)|x|^{-\alpha}$ in Case 2. 
Let us calculate the right-hand side of \rf{pochodna} at these functions. 

\noindent In Case 1, by formula \rf{stala}  
and \rf{lap-alpha-gamma},  we obtain that for $2\alpha<d$  
\begin{equation*}
\begin{split} 
-\langle (-\Delta)^{\alpha/2}\un_{B_{R_{t_0}}}, |x|^{-\g}\rangle 
&= -\left\langle \un_{B_{R_{t_0}}}, (-\Delta)^{-\alpha/2}|x|^{-\g}\right\rangle\nonumber\\  
 &= -\left\langle \un_{B_{R_{t_0}}},\tilde C_{\g,\alpha}\frac{d-2\alpha}{d-\alpha}s(\alpha,d)   |x|^{-(\alpha+\g)} \right\rangle \nonumber\\   
 &= -\tilde C_{\g,\alpha}\s_d \frac{d-2\alpha}{d-\alpha}s(\alpha,d)\int_0^{R_{t_0}} r^{d-(\alpha+\g)-1}\dr \nonumber\\
 &=- \tilde C_{\g,\alpha}\frac{\s_d(d-2\alpha)}{(d-\alpha)(d-\alpha-\g)} s(\alpha,d)R_{t_0}^{d-(\alpha+\g)},\nonumber
\end{split}
\end{equation*}
 where   the constants 
 $$\tilde C_{\g,\alpha}=\frac{\Gamma\left(\frac{d-\gamma}{2}\right)\Gamma\left(\frac{\alpha+\gamma}{2}\right)}{\Gamma\left(\frac{d-\alpha-\gamma}{2}\right)\Gamma\left(\frac{\alpha}{2}\right)} \times \frac{\Gamma\left(\frac{d}{2}-\alpha\right)\Gamma\left(\frac{\alpha}{2}\right)}{\Gamma\left(\frac{d-\alpha}{2}\right)\Gamma(\alpha)} $$ 
obviously satisfy  $\lim_{\g\rightarrow \alpha}\tilde C_{\g,\alpha}=1$. 

 Thus, applying the third of equalities leading to  \rf{pochodna} with $u(x,t)=K\frac{d-\g}{\s_d}|x|^{-\g}$ and the corresponding  $z(r,t)=K$, we get 
 \begin{equation}
 \begin{split}
 \frac{\partial}{\partial t}z(R_{t_0},t){\big|_{t=t_0}}
\le& -\langle(-\Delta)^{\alpha/2}\un_{B_{R_{t_0}}},u\rangle +\frac{d-\g}{\sigma_d}R_{t_0}^{d-2\g}z^2\\
=&- \tilde C_{\g,\alpha}\frac{\s_d(d-2\alpha)K}{(d-\alpha)(d-\alpha-\g)}s(\alpha,d)\frac{d-\g}{\s_d}R_{t_0}^{d-(\alpha+\g)}\\
&\qquad +
\frac{d-\g}{\sigma_d}R_{t_0}^{d-2\g}\frac{(K\s_d)^2}{(d-\g)^2}.
\end{split}
\label{3stars}
\end{equation} 
Now, we look at the sign of the right-hand side in \rf{3stars} in Case 1 and in Case 2 separately. 

\noindent  In Case 1, we have the inequality 
$KR_{t_0}^{\alpha-\g}\le \epsilon s(\alpha,d)\frac{d-\g}{d-\alpha}$ which holds true  if $R_{t_0}\le R_\#$.  
Thus, for $\g=\alpha$,  the inequality $K<\eps\frac{d-\alpha}{\s_d}s(\alpha,d)$ follows, hence the estimate 
$$
\frac{K}{d-\alpha}R_{t_0}^{d-2\alpha}\big(-(d-\alpha)s(\alpha,d)+K\s_d\big)<0
$$
holds. 
By the continuity argument applied to \rf{3stars}, there exists $\g_0>0$ such that for all $\g\in(\g_0,\alpha]$, we still have $\frac{\partial}{\partial t}z(R_{t_0},t){\big|_{t=t_0}}<0$. 

\noindent  In Case 2, we use the function $u(x,t)=\eps s(\alpha,d)|x|^{-\alpha}$ and $\g=\alpha$ in the calculations leading to \rf{3stars}, which gives the following counterpart of inequality \rf{3stars} 
$$
\frac{\partial}{\partial t}z(R_{t_0},t){\big|_{t=t_0}}\le \s_d\frac{1}{d-\alpha}R_{t_0}^{d-2\alpha}\eps s(\alpha,d)^2(-1+\eps)<0
$$
for each $\eps\in(0,1)$. 
This completes the proof of Theorem \ref{g2}. 
\qed

%%%%%%%%%%%%%%%%%%%%%%%%%%%%%%%%
\section
{Global-in-time solutions for $\alpha\in(0,1)$} \label{smooth-sol}

We are in a position to prove Theorem \ref{mainglobth} and we  proceed in the usual way: we construct local-in-time solutions which can be then extended globally in time due to the comparison principle proved in Section \ref{5}.

First, we consider a doubly regularized (a parabolic regularization together with a~smoothing of the  nonlinearity) counterpart of problem  \rf{equ}--\rf{ini}
\begin{align}
\label{reg-reg} u_t+(-\Delta)^{\alpha/2}u+\nabla\cdot(u\fie\ast\nabla v)&=\delta\Delta u,\\
\qquad \Delta v+u&=0,
\\
u(x,0)&=u_0(x),\label{ini-reg}
\end{align}
with a constant  $\delta>0$ and where $\fie$ is a smooth approximation of the Dirac $\delta_0$ measure ({\em e.g.} $\fie(x)=\e^{-d}\phi\left({|x|}/{\e}\right)$ with $0\le \phi\in{\mathcal C}_c^\infty(\R)$, $\phi(x)=1$ for $|x|\le 1$), $\e>0$,
$\int \phi\dx=1$.

\begin{lemma}\label{lem:reg1} 
Suppose that a function $u_0\ge 0$ is  radial and satisfies 
 $u_0\in W^{4,n}(\R^d)\cap L^1(\R^d)$ with the exponent $n=2p>d/4$ for some $p\in \N$. 
Then,  problem  \rf{reg-reg}--\rf{ini-reg}  supplemented with such an initial condition possesses a unique nonnegative, radial solution $u=u_{\e,\delta}\in {\mathcal C^{1}}([0,T_{\e,\delta}], W^{4,n}(\R^d)\cap L^1(\R^d))$ on a time interval with $T_{\e,\delta}>0$.  
\end{lemma}
\begin{proof}
A construction of such local-in-time solutions is standard and it can be based on  the Duhamel formula ({\em  cf.} \rf{Duh} below) written for the initial-value problem for system \rf{reg-reg}--\rf{ini-reg}. See Section \ref{smooth-sol-} for a counterpart of such reasoning. 
\end{proof}
 
Note, however, that the length of the interval of the existence of the solution
constructed in Lemma  \ref{lem:reg1}  depends on $\e,\delta$.  
The following lemma implies immediately  that such a~solution can be continued to a common interval $[0,T_0]$ with $T_0>0$ independent of $\e>0$ and of $\delta>0$.

 \begin{lemma}\label{apr}
Let $A>0$, $s\in{\mathbb N}$, $p\in{\mathbb N}$ be such that $2p>d/s$.
 Consider a solutions   $u\in{\mathcal C}([0,T],W^{s,2p}(\R^d))$  
of the regularized problem \rf{reg-reg}--\rf{ini-reg}.

\begin{itemize}

\item[(i)]
 Then, for all $t\in [0,T]$,  the  inequality 
\be
\|u(t)\|^{2p}_{W^{s,2p}}\le \|u_0\|^{2p}_{W^{s,2p}}+C\int_0^T\|u(\tau)\|^{2p+1}_{W^{s,2p}}\dta\label{s-p}
\ee
holds true with a constant  $C$ independent of  $\e,\delta$. 
In particular, for an initial datum satisfying 
 $\|u_0\|_{W^{s,2p}}\le A$, we have  the estimate $\|u(t)\|_{W^{s,2p}}\le 2A$  for all $t \in [0,{1}/(4AC)]$.

\item[(ii)] Let $\|u\|_{{\mathcal C}([0,T]\times \R^d)}$ be  controlled {\em apriori}. Then, there exist increasing functions $C_{s,A}(t)<\infty$ on $[0,\infty)$  such that inequality 
\be
\|u(t)\|^{2p}_{W^{s,2p}}\le\|u_0\|^{2p}_{W^{s,2p}} +\int_0^tC_{s,A}(\tau)\|u(\tau)\|^{2p}_{W^{s,2p}}\dta\label{Ws-p}
\ee
holds for $t\in[0,T]$.
\end{itemize}
 \end{lemma}
 
  \begin{proof}  Here, we denote 
  $\p^\beta =\left(\frac{\partial^{\beta_1}}{\partial x_1^{\beta_1}}, \dots , \frac{\partial^{\beta_d}}{\partial x_d^{\beta_d}}\right) $ for each $\beta=(\beta_1, \dots , \beta_d)\in \N^d$, $|\beta|\le s$. 
  
{\it Item (i).}  Using equation \rf{reg-reg} and the Leibniz rule (skipping the integrals of good sign coming from all the diffusion terms) we have got
 \begin{equation}\label{Lei}
 \begin{split}
\frac{\rm d}{\dt}\int(\p^\beta u)^{2p}\dx&=\int(\p^\beta u)^{2p-1}\p_t\p^\beta u \dx \le\int(\p^\beta u)^{2p-1}\p^\beta(\fie\ast\nabla v)\dx \\
&\approx \sum_j\int(\p^\beta u)^{2p-1}\nabla(\p^ju)\cdot\nabla(\p^{\beta-j}\fie\ast v)\dx.
\end{split}
\end{equation}
Here, to obtain second inequality, we have skipped  integrals of good sign coming from all the diffusion terms because
$$
-\int(\p^\beta u)^{2p-1}(-\Delta)^{\alpha/2} (\p^\beta u) \dx\leq 0
$$
for each $p\ge 1$ and $\alpha\in (0,2]$ by the Stroock--Varopoulos inequality, see {\it e.g.} \cite[Prop. 3.1]{BIK}.    
Now, we estimate  the terms for $ j<\beta$ and $j=\beta$, separately. 

For $ j<\beta$, we use the decomposition  $\fie\ast v=K_0\ast u+K_\infty\ast u$ with the kernels 
\be
K_0(x)=\phi(x)(\fie\ast\nabla(-\Delta)^{-1})(x),\ \ K_\infty(x)=(1-\phi(x))(\fie\ast\nabla(-\Delta)^{-1}(x)).\label{kern}
\ee
Evidently, we have $K_0\in L^q(\R^d)$ if  $q<\frac{d}{d-1}$. 
Thus, for $u\in L^\infty$, we obtain  that $K_0\ast u\in L^\infty$ together with 
the estimate  
$$
\|K_\infty\ast \p^\beta u\|_\infty \le \|\p^\beta K_\infty\|_\infty\|u\|_1\nonumber
$$
for each multiindex $\beta=(\beta_1, \dots , \beta_d)$ with  integers $\beta_i\ge 0$. 
By an elementary argument, we can also show that  $K_0\ast u\in L^\infty$ if $u\in W^{s,2p}$ with  suitably large $s$, $p$. 

For $j=\beta$, we proceed as follows   
  \bea
\left|\int(\p^\beta u)^{2p-1}\nabla(\p^\beta u)\cdot\nabla( \fie\ast v)\dx\right| &=&\left|\int\nabla(\p^\beta u)^{2p}\cdot\nabla(\fie\ast v)\dx\right|\nonumber\\
&\le& \left|\int(\p^\beta u)^{2p}\Delta(\fie\ast v)\dx\right|\nonumber\\
&=&\left|\int(\p^\beta u)^{2p}\fie\ast u\dx\right|.\nonumber
\eea
Using the estimate  $\|u\|_\infty\le C\|u\|_{W^{s,n}}$  (valid  for $s>d/n$) in  the computations above, we obtain the inequality 
$$
\frac{\rm d}{\dt}\|u\|_{W^{s,n}}^n\le C\|u\|_{W^{s,n}}^{n+1},
$$
and thus 
\be
\|u(t)\|_{W^{s,n}}^n\le \|u_0\|_{W^{s,n}}^n +C\int_0^t\|u(\tau)\|^{n+1}_{W^{s,n}}\dta.\label{wzor1}
\ee
Now, if $\sup_{0\le t\le T}\|u(t)\|_{W^{s,n}}\le 2A$, then 
$$
\sup_{0\le t\le T}\|u(t)\|_{W^{s,n}}^n\le \|u_0\|_{W^{s,n}}^n+C(2A)^{n+1}t \qquad \text{for all} \quad t\in [0,T].$$ 
Next, choosing  $T>0$ as the first moment when $\|u(T)\|_{W^{s,n}}= 2A$, we obtain the estimate 
$$
(2A)^n\le A^n+CT(2A)^{n+1}
$$
which gives immediately  that $T\ge \frac{1}{4AC}$. 
%\medskip
 
 {\it Item (ii).}  Let us estimate again a generic term in \rf{Lei} 
obtained from the Leibniz formula 
 $$
 J_{\beta,j}=\left|\int(\p^\beta u)^{2p-1}\p^{\beta-j}\nabla u\cdot\nabla(\p^j v) \dx\right|.
 $$ 
Here, the assumption on the radial symmetry of $u$ is crucial 
because by  Lemma~\ref{potential}, we have got the equations
$$\p_i \p_kv=\p_i \p_k(-\Delta)^{-1}u =\p_i\left(\frac{x_k}{|x|^d}M(|x|)\right)$$ obtained from equality $\nabla v(x)\cdot x=-\frac1{\sigma_d}|x|^{2-d}M(|x|)$ (see Lemma \ref{potential}) with $u(x)=\frac{1}{\sigma_d}R^{1-d}M'(R)$, $|x|=R$ (this is identity \rf{u-rad}), are bounded since $u$ is {\em apriori} ${\mathcal C}^2$. %\notka{Uproscic to zdanie.}

For $s=1$, $|j|=1$, we have $|\p_i\p_k v|\le C(\|u\|_\infty+\|u\|_1)$, and consequently $J_{\beta,j}\le C\|u\|^{2p}_{W^{s,2p}}(\|u\|_\infty+\|u\|_1)$.

For $|j|=1$, $s>1$, we recall by formulas \rf{kern} that $\nabla(\p v)=\nabla(-\Delta)^{-1}\p u=K_0\ast\p u+K_\infty\ast\p u$. 
By the recurrence assumption $\p u\in L^{2p}$ with the norm bounded by $C_{1,A}(t)$, so that $\|K_0\ast \p v\|_{ L^\infty}\le CC_{1,A}(t)$. 
Similarly, we get $|K_\infty\ast\p v|=|\p K_\infty\ast u|\le \|u\|_1\|\p K_\infty\|_\infty$.

For $|j|\ge 2$, $s\ge 2$, by recurrence, we infer that $\|\nabla\p^jv\|_{W^{|j|-1,2p}}\le C_{s-1}(t)$ and 
$$\|\p^{\beta-j}\nabla(-\Delta)^{-1}u\|_\infty\le C\|u\|_{W^{s-1,2p}},$$ 
and we are done. 
 \end{proof}
% \bigskip

In the following lemma, we pass to the limit $\e\to 0$ and $\delta\to 0$ in the regularized problem \rf{reg-reg}--\rf{ini-reg}.

 \begin{lemma}\label{let}
 Let $\alpha<1$, $n=2p>d+1$,  $p\in {\mathbb N}$, and  $A>0$.
For every $u_0\in L^1(\R^d)\cap W^{4,n}(\R^d)$ such that  $u_0\ge 0$ and $\|u_0\|_{ W^{4,n}(\R^d)}\le A$, 
 there exists a solution $u $ of problem \rf{equ}--\rf{ini}  the function  $u$ is defined on $[0,T_0],$  where $T_0= 1/(4AC)$ is defined in Lemma~\ref{apr}.
 This solution satisfies
\be 
u \in{\mathcal X}_{T_0}\equiv{\mathcal C}([0,T_0],W^{4,n}(\R^d))\cap {\mathcal C}^1([0,T_0],W^{3,n}(\R^d)).\label{space-X}
\ee
Moreover, we have $\sup_{0\le t\le T_0}\|u(t)\|_{W^{4,n}(\R^d)}\le 2A$.
 \end{lemma}
 
\begin{proof}
Let $\|u_0\|_{W^{s,n}}\le A$ for some $A>0$ and $n\in{\mathbb N}$. 
Suppose that a solution $u_{\e,\delta}$ exists on the interval $[0,kT_{\e,\delta}]$, $k\in{\mathbb N}$, with $u_{\e,\delta}\in{\mathcal C}([0,kT_{\e,\delta}], W^{4,n}(\R^d)\cap L^1(\R^d))$, $T_{\e,\delta}>0$ being the common existence time for $\|u_0\|_{W^{s,n}\cap L^1}\le 2A$ with $\|u(.,{T_{\e,\delta}})\|_{W^{s,n}\cap L^1}\le 4A$.  
 If $\sup_{0\le s\le kT_{\e,\delta}}\|u(.,s)\|_{W^{4,n}\cap L^1}\le 2A$, then this solution can be continued onto $[0,(k+1)T_{\e,\delta}]$. 
 By Lemma \ref{apr} we have  $\sup_{0\le s\le kT_{\e,\delta}}\|u(.,s)\|_{W^{4,n}\cap L^1}\le 2A$  for $ kT_{\e,\delta}\le T_0=\frac{1}{2CA}$  so {\em independently} of $\e>0$,$\delta>0$. Assume $2p>d+1$.
 By compactness, we are able to extract a subsequence $u_{\e_j}(.,t)\in W^{3,n}\cap L^1$ which is in $\mathcal C^2$, and the limiting function solves system \rf{equ}--\rf{eqv} with $u_0\in W^{4,n}(\R^d)\cap L^1(\R^d)$. 
 \end{proof} 

We also need a technical lemma on a decay property of  radial solutions.  
  
  \begin{lemma}\label{x-decay}
  Suppose that $u=u(x,t)\ge 0$ is a radial solution of system \rf{equ}--\rf{eqv} with $u_0$ satisfying bound \rf{zalo-1} with a sufficiently small $\eps>0$. 
 Moreover, suppose that $u$ satisfies $\int u(x,t)\dx \le M$, and the estimates 
\be
\|\partial^\beta u(t)\|_p+\|u(t)\|_p\le C(t),\label{Sob1}
\ee
with $|\beta|\le n$, some $C(t)$ and a sufficiently large fixed $n$. 
Then $\lim_{x\to\infty}|x|^\alpha u(x,t)=0$ uniformly  on each interval
  $[0,T]$, $T>0$.   
  \end{lemma}
  
\begin{proof} 
The estimate \rf{Sob1} is, in fact, satisfied for sufficiently smooth solutions, {\em e.g.} those constructed either  in this Section \ref{smooth-sol}  or in \cite{LRZ}. 
Let $\Psi\ge 0$ be a smooth bump function supported on an annulus:  ${\rm supp\,}\Psi\subset\{1\le |x|\le 2\}$, and its scaling $\Psi_R(x)=\Psi\left(\frac{x}{R}\right)$, $R>0$. Define the moment of $u$ by 
\be
\Lambda_R(t)=\int\Psi_R(x)u(x,t)\dx. \label{lambda}
\ee
Computations similar to those in the proof of Theorem \ref{blow} in Section \ref{bl}  lead to the bound
$$
\left|\frac{\rm d}{\dt}\Lambda_R(t)\right|\le \frac{C(M)}{R^\alpha}.
$$  
This, in turn, gives 
$$
\Lambda_R(t)\le \frac{C(M,t)}{R^\alpha}, 
$$ 
which by radial symmetry  implies that 
\be
\int_{\{|x-x_0|\le 1\}}u(x,t)\dx\le \frac{C_1}{R^{\alpha+d-1}},\label{war-1}
\ee
when  $|x_0|=R>0$. Indeed, the spherical shell of radius $\approx R$ and of width $\approx 1$ contains $\approx R^{d-1}$ unit balls. 

On the other hand, 
\be
\int_{\{|x-x_0|\le 1\}}\left(\left|\partial^\beta u(t)\right|^p+u(t)^p\right)\dx \le C_2\label{war-2}
\ee 
for $|\beta|\le n$ with a sufficiently big $n$. 
The condition \rf{war-2} implies now that 
\be
\sup_{\{|x-x_0|\le1\}}\left(|\nabla u(x,t)|+|\partial_{x_i}\partial_{x_k} u(x,t)|\right)\le C_3.\label{pochodne}
\ee
Next, we consider the truncation $\chi (x-x_0)u(x,t)$ where $\chi\ge 0$ has its support in the unit ball. If for some $x_1$ with $|x_1-x_0|\le 1$ \ \ 
$$\max_{x\in\R^d}\chi(x-x_0)u(x,t)=\chi(x_1-x_0)u(x_1,t),$$
then, denoting again by $u$ the function $\chi u$, from inequalities \rf{war-1} and \rf{pochodne} we obtain 
$$
|u(x,t)-u(x_1,t)|\le C|x-x_1|^2\le \frac12 u(x_1,t).
$$
Indeed, if $u(x_1,t)\ge \frac{3}{C}$, then 
$$
\frac12 u(x_1,t)\le \int_{\{|x-x_1|\le 1\}}u(x,t)\dx\le \frac{C}{R^{d-1+\alpha}},
$$ and we are done.
Otherwise, if $u(x_1,t)< \frac{3}{C}$, then 
$$
C_1u(x_1,t)^{d/2+1}\le  \int_{\{|x-x_1|\le 1\}}u(x,t)\dx\le \frac{C}{R^{d-1+\alpha}}.
$$
In both the cases, for $\alpha<1$ we get the conclusion since the inequality $\frac{d-1+\alpha}{\frac{d}{2}+1}\ge \alpha$ is satisfied for $\alpha\le 2\frac{d-1}{d}$. 
\end{proof}

\begin{proof}[Proof of Theorem  \ref{mainglobth}.]
The proof of this  theorem is a standard application of   Lemma \ref{let}, Lemma  \ref{apr} and the pointwise  comparison principle in Theorem \ref{g1}. 

Let us fix $\alpha<1$, $n=2p\ge d+1$, $p\in{\mathbb N}$, $\e<1$. 
By Lemma \ref{let} there exists  $T_0>0$ such that the system \rf{equ}--\rf{eqv} has a solution $u\in {\mathcal X}_{T_0}$, with the space ${\mathcal X}_{T_0}$ defined in \rf{space-X}. 
We will show that this solution can be continued onto the interval $[0,T_1]$, to $u\in {\mathcal X}_{T_1}$, with $T_1-T_0\ge \Delta(A,[T_0],\e)>0$. 
First, observe that by assumptions of Theorem \ref{mainglobth}  and property \rf{space-X}, there exist $K$, $N$ and $\g<\alpha$ such that $0\le u_0(x)< \min\left\{N,\frac{K}{|x|^\gamma},\frac{\eps s}{|x|^\alpha}\right\} $, so that by Lemma \ref{x-decay},   condition \rf{decay_inf} of  Theorem \ref{g1} is satisfied. Consequently the estimate
$$
u(x,t)< \min\left\{N,\frac{K}{|x|^\gamma},\frac{\eps s}{|x|^\alpha}\right\}
$$
holds for each $t\in[0,T_0]$. 
In particular, by Lemma \ref{apr} we infer that $\|u(t)\|_{W^{4,n}}\le H(A,[T_0]+2,\e)$. 
Take  $T'<T_0$, close to $T_0$. By Lemma \ref{let}, the solution $v$ with the initial condition $v(0)=u(T')$ exists on an interval of length (at least) $\Delta=\Delta(A,[T_0],\e)$.
Therefore, the solution of the original Cauchy problem can be continued onto $\left[0,T_0+\frac12\Delta\right]$, which shows the claim. 
\end{proof}  

%%%%%%%%%%%%%%%%%%%
 \section{Unique global-in-time solutions for $\alpha \in (1,2)$} \label{smooth-sol-} 
 
 In this section, we prove Theorem  \ref{2globth} by constructing global-in-time solutions in  the homogeneous Morrey   space $ M^p(\R^d)$.
 Let us begin with   auxiliary estimates. 

%%%%%%%%%%%%%%%%%%%5
\begin{proposition}\label{prop:Mor} 
There exists a constant $c(d)\in (0,1)$ such that for each nonnegative and  radially symmetric $v\in M^{d/\kappa}(\R^d)$ with $\kappa\in[1,d)$ we have the inequality 
$$
c(d) \mn v\mn_{M^{d/\kappa}}\leq  \sup_{R>0}R^{\kappa-d}\int_{\{|y|<R\}}v(y)\dy
\leq  \mn v\mn_{M^{d/\kappa}}.
$$ 
\end{proposition}

\begin{proof} 
The second inequality results immediately from the definition of the norm in $M^{d/\kappa}(\R^d)$, see \rf{hMor}.
For the proof of the first inequality, we fix $x_0\in\R^d\setminus\{0\}$ and $\r\in(0,R)$ where $R=|x_0|$. 
By comparison of the volumes, one can prove that the spherical shell $\{ R-\varrho\le |x|\le R+\varrho\}$ contains at least $C(d)\left(\frac{R}{\r}\right)^{d-1}$ disjoint balls of radius $\varrho$, where $C(d)>0$ is a number depending on the dimension $d$ only.
Thus, by the radial symmetry of $v$, we obtain 
$$
\int_{\{|x-x_0|<\varrho\}}v(x)\dx\le c(d)\left(\frac{\varrho}{R}\right)^{d-1}\int_{\{|x|<R\}}v(x)\dx 
$$ 
with another constant $c(d)>0$. 
Consequently,  since $\kappa\in[1,d)$ and $0<\r<R$, we have the estimate 
$$
\sup_{0<\varrho\le R}\varrho^{\kappa-d}\int_{\{|x-x_0|<\varrho\}}v(x)\dx\le c(d)R^{\kappa-d}\int_{\{|x|<R\}}v(x)\dx
$$
for each $x_0\in\R^d$. 
Computing the upper bound with respect to $x_0\in\R^d$ and $R>0$, we complete the proof of the first inequality. 
 \end{proof}
  
%%%%%%%%%%%%%%%%%%%%%
%\bigskip

 Since we assume that $\alpha\in (1,2)$, we do not need to construct local-in-time solutions via the regularized problem  \rf{reg-reg}--\rf{ini-reg}. 
We present here a standard construction,   {\em cf. e.g.} \cite{B-SM,K-JMAA,Lem} for related computations, which work in the  subcritical case  $\alpha\in(1,2)$ in suitable  Morrey spaces.  
Here, by a  solution, we understand  the  {\em mild} solution of problem \rf{equ}--\rf{ini} which  satisfies the Duhamel formula 
\be
u(t)=T_\alpha(t)u_0+  B(u,u)(t),\label{Duh}
\ee
with the bilinear form 
\be
B(u,w)(t)=\int_0^t \nabla T_\alpha(t-s)(u\nabla E_d*w)(s)\ds.\label{B-form}
\ee
Here 
\be
T_\alpha(t)={\rm e}^{-t(-\Delta)^{\alpha/2}}\label{semi}
\ee
 denotes the semigroup generated by the fractional Laplacian  $-(-\Delta)^{\alpha/2}$ on $\R^d$, and $P_t$ is its integral kernel: $T_\alpha(t)u_0 =P_t\ast u_0$. 
 The function $P_t$ is  of selfsimilar form 
\be
P_t(x)=t^{-d/\alpha}P\left(\frac{x}{t^{1/\alpha}}\right)\ge 0.\label{P}
\ee
It is well known (see {\em e.g.} \cite[Ex. 3.9.17]{J1}) that for $\alpha<2$ the function $P$ has an algebraic decay at infinity 
$$ 0\le P(y)\le C(1+|y|^{d+\alpha})^{-1}$$ 
and $\|P_t\|_1=1$. 
We also need estimates for  $\nabla P_t$  which has the form (see {\em e.g.} \cite{BJ})
\be
\nabla P_t(x)=t^{-(d+1)/\alpha}G\left(\frac{x}{t^{1/\alpha}}\right)\label{G}
\ee
for some smooth function $G$, and satisfies the relations 
\be
\|\nabla P_t\|_1=Ct^{-1/\alpha},\label{L1-G}
\ee
\bea
|\nabla P_t(x)|&\le& Ct^{-(d+1)/\alpha},\label{G-inf}\\
|\nabla P_t(x)|&\le& Ct^{1/\alpha-1}|x|^{-(d+\alpha+1)}.\label{G-point}
\eea  
In the sequel, we will use the estimates for the semigroup $T_\alpha(t)$ and for its gradient  acting in the Morrey spaces similar to those for the action in the Lebesgue spaces. 
These are analogous to the estimates for the heat semigroup for $\alpha=2$ in \cite[Prop. 3.2]{G-M}, \cite[Th. 3.8, (3.71)--(3.75), (4.18)]{Tay} recalled in \cite[(13)--(14)]{B-SM}, and can be, {\em e.g.},  obtained using inequalities \rf{L1-G}--\rf{G-point}. % together with the H\"older inequality. 

For $1\leq p\leq p_2\leq \infty$  
\be
\mn T_\alpha(t)f\mn_{M^{p_2}}\le Ct^{-d(1/p-1/p_2)/\alpha}\mn f\mn_{M^p}\label{polgrupa}
\ee
holds.
Moreover,  for $1\le p_1\leq p_2\le\infty$,  the estimate for the gradient of the semigroup reads 
\be
\mn \nabla T_\alpha(t)f\mn_{M^{p_2}}\le Ct^{-1/\alpha-d(1/p_1-1/p_2)/\alpha}\mn f\mn_{M^{p_1}}.\label{grad-polgrupy}
\ee
We also recall from \cite[Prop. 3.1]{G-M} a version of estimates of the Riesz potential in the Morrey norms. For $E_d$ being a fundamental solution of 
$-\Delta$ in $\R^d$ with $d\geq 2$, we have the estimates
\begin{equation}\label{Sob0}
\mn \nabla E_d* u\mn_{M^r} \leq C\mn u\mn_{M^p}\quad \text{with} \quad 
\frac{1}{r} =\frac{1}{p}-\frac{1}{d}
\end{equation}
as well as 
\be
\| \nabla E_d* u\|_\infty\le C\mn u\mn_{M^{p}}^\mu\mn u\mn_{M^r}^\nu\label{Sob}
\ee
with $1\le p<d<r$ and $\mu=\frac{\frac1d-\frac1r}{\frac1p-\frac1r}$, $\nu=\frac{\frac1p-\frac1d}{\frac1p-\frac1r}$ so that $\mu+\nu=1$. Below, we shall only use  the following particular version of inequality \eqref{Sob}:
\be
\| \nabla E_d* u\|_\infty\le C\mn u\mn_{M^{d/\alpha}}^{1/\alpha}\| u\|_\infty^{1-1/\alpha}\label{Sob1}
\ee

We are ready to formulate and prove a local-in-time existence result which is valid even without radial symmetry assumption on $u_0$.

\begin{proposition} \label{lok-istn}
Given $u_0\in M^{d/\alpha}(\R^d)\cap L^\infty (\R^d)$ with $d\ge 2$, there exist $T=T(u_0)>0$ and a unique local-in-time  mild solution 
$$
u\in {\mathcal X}_T\equiv{\mathcal C}([0,T], M^{d/\alpha}(\R^d)\cap L^\infty (\R^d))
$$ 
of problem \rf{equ}--\rf{eqv} with $\alpha\in (1,2]$.
\end{proposition}

\begin{proof}
We supplement the space ${\mathcal X}_T$ with the usual norm
$$
\|u\|_{{\mathcal X}_T}\equiv \sup_{t\in[0,T]}\mn u(t)\mn_{M^{d/\alpha}}+
\sup_{t\in[0,T]}\| u(t)\|_{\infty}
$$ 
and we shall  find the solution of equation \eqref{Duh}--\rf{B-form} by the Banach fixed point theorem. First, we note that by inequalities \rf{polgrupa}, we have got
$
\|T_\alpha(\cdot )u_0\|_{\mathcal{X}_T}\leq C \|u_0\|_{\mathcal{X}_T}
$
with a constant independent of $T$ and of $u_0$. 
Next, we estimate the bilinear form \rf{B-form} in the norm of the space 
${\mathcal X}_T$. For all $u,w\in {\mathcal X}_T$, by inequalities 
\rf{grad-polgrupy} and \rf{Sob1}, we obtain the estimates
\begin{equation}
\begin{split}
\mn B(u,w)\mn_{M^{d/\alpha}}&
\le C\int_0^t(t-s)^{-1/\alpha}\mn u(s)\mn_{M^{d/\alpha}}\|\nabla E_d*w(s)\|_\infty\ds\nonumber\\
&\le C\int_0^t(t-s)^{-1/\alpha}\mn u(s)\mn_{M^{d/\alpha}}\mn w(s)\mn_{M^{d/\alpha}}^{1/\alpha}\| w(s)\|_{\infty}^{1-1/\alpha}\ds\nonumber\\
&\le CT^{1-1/\alpha} \|u\|_{\mathcal{X}_T}\|w\|_{\mathcal{X}_T}
\end{split}
\end{equation}
and, analogously,
\begin{equation}
\begin{split}
\| B(u,w)\|_{\infty}&
\le C\int_0^t(t-s)^{-1/\alpha}\| u(s)\|_{\infty}\|\nabla E_d*w(s)\|_\infty\ds\nonumber\\
&\le C\int_0^t(t-s)^{-1/\alpha}\| u(s)\|_{\infty}\mn w(s)\mn_{M^{d/\alpha}}^{1/\alpha}\| w(s)\|_{\infty}^{1-1/\alpha}\ds\nonumber\\
&\le CT^{1-1/\alpha} \|u\|_{\mathcal{X}_T}\|w\|_{\mathcal{X}_T}.
\end{split}
\end{equation}
The usual reasoning (see {\it e.g.}~\cite{B-SM,K-JMAA,Lem}) completes the 
construction of a unique solution in the space $\mathcal{X}_T$ for sufficiently small $T$ depending on $u_0$.
\end{proof}

\begin{proof}[Proof of Theorem  \ref{2globth}]
A local-in-time solution is constructed in Proposition \ref{lok-istn}. 
Since $\alpha>1$, this solution is sufficiently regular ({\em e.g.} $u\in{\mathcal C}^2(\R^d\times[0,T))$) which can be proved repeating the reasoning from \cite{DGV}.
This solution is radial and nonnegative  if the corresponding initial datum is so, by a usual comparison argument. 
To prove that this local-in-time solution can be extended to all $t>0$, it suffices to show that neither 
$\mn u(t)\mn_{M^{d/\alpha}}$ nor $\|u(t)\|_\infty$ can blowup in a finite time.

By assumptions Theorem \ref{2globth} and by Remark \ref{B},
 there exist constants $K>0$ and $\g\in(0,\alpha)$ such that 
$$
\int_{\{|x|<R\}}u_0(x)\dx<\min\left\{ KR^{d-\g},\eps \frac{\s_d}{d-\alpha}s(\alpha,d)R^{d-\alpha}\right\}\ \ {\rm for\ all\ } R>0,
$$
Then, applying  Proposition \ref{prop:Mor}, one can  immediately check that  $u_0\in M^p(\R^d)$ with  $p=d/\g>d/\alpha$.
%see Remark \ref{B} for more detail. 
Thus, by the comparison principle proved in Theorem \ref{g2} combined with Proposition \ref{prop:Mor}, there exists a number $C$ independent of $T$ such that $|\!\!| u(t)|\!\!|_{M^{d/\alpha}}\le C$ and $|\!\!| u(t)|\!\!|_{M^p}\le C$ for all $t\in[0,T]$.

Next, we  estimate the $L^\infty$-norm of both sides of equation \rf{Duh} 
using inequalities \rf{grad-polgrupy} and \rf{Sob0} with $1/r=1/p-1/d$ 
in the following way 
\begin{equation*} 
\begin{split}
\|u(t)\|_\infty&\leq \|T_\alpha(t)u_0\|_\infty
+C \int_0^t (t-s)^{-\frac{1}{\alpha}-\frac{d}{\alpha r}}\|u(s)\|_\infty \mn\nabla E_d*u(s)\mn_{M^r}\ds\\
&\leq \|u_0\|_\infty + C
\int_0^t (t-s)^{-\frac{1}{\alpha}-\frac{d}{\alpha}
\left(\frac{1}{p}-\frac{1}{d}\right)}
\|u(s)\|_\infty \mn u(s)\mn_{M^p}\ds.
\end{split}
\end{equation*}
Thus,  the $L^\infty$-norm of the solution  is controlled locally in time thanks to a singular Gronwall type argument ({\em cf.} \cite[1.2.1, 7.1.1]{H}), because
$$
-\frac{1}{\alpha}-\frac{d}{\alpha}
\left(\frac{1}{p}-\frac{1}{d}\right)\in (-1,0)
\quad \text{for} \quad p>\frac{d}{\alpha}
$$
and because 
$\sup_{s>0}  \mn u(s)\mn_{M^p}<\infty$ by Theorem \ref{g2} combined with Proposition \ref{prop:Mor}.
\end{proof}
%\bigskip

%%%%%%%%%%%%%%%%%%%%%%%%%%%%%%%
\section{Blowup of radially symmetric solutions } \label{bl}

In this section we prove Theorem \ref{blow} using the method of truncated moments which is reminiscent of that in the papers  \cite{N1,K-O,BKZ} for $\alpha=2$ and in our recent papers  \cite{BKZ,BKZ-NHM,BCKZ}, adjusted to the case $\alpha<2$.  
First, we define a continuous bump function $\psi$ and its rescalings for $R>0$ 
\be
\psi(x)=(1-|x|^2)_+^{1+\alpha/2}
=\left\{
\begin{array}{ccc}
(1-|x|^2)^{1+\alpha/2}& \text{for}& |x|<1,\\
0& \text{for}& |x|\geq 1,
\end{array}
\right.
\qquad  \psi_R(x)=\psi\bigg(\frac{x}{R}\bigg). \label{bump}
\ee
The function $\psi$ is piecewise ${\mathcal C}^2(\mathbb R^d)$, with its support ${\rm supp}\,\psi=\{|x|\le 1\}$,   and  satisfies 
\be
\nabla\psi(x)=-(\alpha+2)x(1-|x|^2)_+^{\alpha/2}.\label{gradpsi}
\ee 
The action of the fractional powers $(-\Delta)^{\alpha/2}$ of the Laplacian operator on functions like $(1-|x|^2)_+^\kappa$ leads to formulas involving hypergeometric functions. 
In the particular case $\kappa=1+\alpha/2$, it follows from \cite[p. 39]{MOS} 
that this is a linear polynomial in $|x|^2$
\be
(-\Delta)^{\alpha/2}\psi(x)=m_\alpha\left(1-\frac{d+\alpha}{d}|x|^2\right)\ \ \ {\rm on}\ \ \ \ {\{|x|\le 1\}},\label{delta-psi}
\ee 
with the constant 
\be
m_\alpha=2^\alpha\Gamma\left(2+\frac{\alpha}{2}\right)\frac{\Gamma\left(\frac{d+\alpha}{2}\right)}{\Gamma\left(\frac{d}{2}\right)}
\sim 2^{\alpha/2}\left(1+\frac{\alpha}{2}\right)\Gamma\left(1+\frac{\alpha}{2}\right)d^{\alpha/2}\ \ {\rm as\ } d\to\infty.\label{m}
\ee 
 The  relation used to obtain asymptotics of $m_\alpha$ 
\be 
\frac{\Gamma\left(z+a\right)}{\Gamma\left(z+b\right)}\sim z^{a-b}\ \ {\rm as\ } z\to\infty,\label{Gamma}
\ee
follows from the Stirling formula. 
Moreover, 
$$(-\Delta)^{\alpha/2}\psi(x)\le 0, \ \ \  |x|\ge 1,$$
holds similarly as was shown in \cite[Lemma 4.3]{BCKZ}, so that we have the inequality 
\be
(-\Delta)^{\alpha/2}\psi(x)\le \ell_\alpha \psi(x),\label{lapl-psi}
\ee 
with 
\be
\ell_\alpha =\frac{2^{\alpha/2}\Gamma\left(1+\frac{\alpha}{2}\right)}{\left(1+\frac{\alpha}{2}\right)^{\alpha/2}}
   \frac{(d+\alpha)^{1+\alpha/2}}{d} 
 \frac{\Gamma\left(\frac{d+\alpha}{2}\right)}{\Gamma\left(\frac{d}{2}\right)} 
 \sim \frac{\Gamma\left(1+\frac{\alpha}{2}\right)}{\left(1+\frac{\alpha}{2}\right)^{\alpha/2}}d^\alpha\ \ {\rm as\ }d\to\infty.\label{ell}
 \ee
Indeed, $\ell_\alpha$ is the least number such that the inequality  $\ell_\alpha(1-s)^{1+\alpha/2}- m_\alpha\left(1-\frac{d+\alpha}{d}s\right)\ge 0$ holds for each $s\in[0,1]$, the minimum of that expression being attained at $s_0=\frac{d-2}{d+\alpha}\in[0,1)$. 

Now, consider  a ``local  moment'' of the solution $u(.,t)$ defined by 
\be
w_R(t)=\int \psi_R(x)u(x,t)\dx\label{moment}
\ee
with the weight function $\psi$ as in \rf{bump}.  
 The evolution of $w_R$ is determined by 
 \bea 
\frac{{\rm d}}{\dt}w_R(t)&=& -\int \psi_R(x)(-\Delta)^{\alpha/2} u(x,t)\dx+\int u(x,t)\nabla v(x,t)\cdot \nabla\psi_R(x)\dx\nonumber\\
&=& -\int  (-\Delta)^{\alpha/2}\psi_R(x)\,u(x,t)\dx\nonumber\\
&\quad&-(\alpha+2)\int u(x,t)\nabla v(x,t)\cdot xR^{-2}\left(1-\left|\frac{x}{R}\right|^2\right)_+^{\alpha/2}\dx\nonumber\\
&\ge& -m_\alpha R^{-\alpha} \int_{\{|x|\le R\}}\left(1-\frac{d+\alpha}{d}\frac{|x|^2}{R^2}\right)u(x,t)\dx\nonumber\\
&\quad&- (\alpha+2)R^{-2}\int_{\{|x|\le R\}} u(x,t)\big(\nabla v(x,t)\cdot x\big)\left(1-\frac{|x|^2}{R^2}\right)^{\alpha/2}\dx, \nonumber
\eea
the second equality followed from the ``integration by parts'' for  $(-\Delta)^{\alpha/2}$. 
Thus, applying inequality \rf{lapl-psi} and Lemma \ref{potential}, we obtain 
\bea 
\frac{{\rm d}}{\dt}w_R(t)&\ge& 
-\ell_\alpha R^{-\alpha}w_R(t)\label{evol}\\ 
&\quad&+\frac{\alpha+2}{\sigma_d}R^{-d}  \int_{\{|x|\le R\}} \left(1-\frac{|x|^2}{R^2}\right)^{\alpha/2}u(x,t)M(|x|,t)\dx.\nonumber
\eea
 Let us write the terms on the right-hand side of inequality \rf{evol} in the radial variables, explicitly. 
 We have  
\be
w_R(t)=R\int_0^1M'(Rr,t)(1-r^2)^{1+\alpha/2}\dr
= (\alpha+2)\int_0^1M(Rr,t)r(1-r^2)^{\alpha/2}\dr,\label{w-r}
\ee
and likewise after the integration by parts 
\begin{equation}
\begin{split}
R\int_0^1 M'(Rr,t)&M(Rr,t)r^{2-d}(1-r^2)^{\alpha/2}\dr\nonumber\\
&=\frac12\int_0^1 M(Rr,t)^2r^{1-d}(1-r^2)^{\alpha/2-1}((d-2)-(d-2-\alpha)r^2)\dr.\nonumber
\end{split}
\end{equation}
Now, the application of the  Cauchy inequality shows that 
{\setlength\arraycolsep{.5pt}
\bea
w_R(t)^2&\le& 2(\alpha+2)\frac{1}{d-2}\int_0^1r^{1+d}(1-r^2)^{\alpha/2}\dr\nonumber\\   
&\qquad& \times \frac{\alpha+2}{2}\int_0^1M^2r^{1-d}(1-r^2)^{\alpha/2-1}((d-2)-(d-2-\alpha)r^2)\dr.\nonumber
\eea}
Therefore, the inequality 
\begin{equation*}
\begin{split}
\frac{{\rm d}}{\dt}w_R(t)\ge&
 -\ell_\alpha R^{-\alpha}w_R(t)\\
& +\frac{(\alpha+2)R^{-d}}{2\sigma_d}\int_0^1 M(Rr,t)^2r^{1-d}(1-r^2)^{\alpha/2-1}((d-2)-(d-2-\alpha)r^2)\dr\nonumber 
 \end{split}
\end{equation*}
 implies 
\be
\frac{{\rm d}}{\dt}w_R(t) \ge w_R(t)\bigg(-\ell_\alpha R^{-\alpha}
+\frac{R^{-d}}{\sigma_d}C(\alpha,d)w_R(t)\bigg) \label{ww}
\ee
for some constant $C(\alpha,d)>0$. 
  For the computation of $C(\alpha,d)$, we used above the relations 
$$(1-r^2)((d-2)-(d-2-\alpha)r^2)^{-1}\le (d-2)^{-1},$$  and  
$$\int_0^1 r^{1+d}(1-r^2)^{\alpha/2}\dr=\frac12 \int_0^1\tau^{d/2}(1-\tau)^{\alpha/2}\dta=
\frac12 \frac{\Gamma\left(\frac{d}{2}+1\right) 
 \Gamma\left(1+\frac {\alpha}{2}\right)} 
{\Gamma\left(\frac{d+\alpha}{2}+2\right)},$$ the latter following from the definition of the Euler Beta function \rf{beta}.  
Now, if initially 
\be
R^{\alpha-d}w_R(0)>\sigma_d\frac{\ell_\alpha}{C(\alpha,d)},\label{www}
\ee
then $w_R(t) $  strictly increases in time, and $w_R(t)$ blows up in a finite time which is a contradiction if $u(x,t)$ is a global-in-time radially symmetric, nonnegative solution.  %becomes greater than the constant total mass $M=\int u(x,t)\dx $ in a finite time which is a contradiction if $u(x,t)$ is a global-in-time radially symmetric, nonnegative, mass conserving solution. 

The proof of Theorem \ref{blow} (i) is complete because $\int_{\{|x|<R\}}u_0(x)\dx\ge w_R(0)$ for each $R>0$.

 Next, under condition \rf{www}, inequality \rf{ww} implies that 
$$
\frac{{\rm d}}{\dt}w_Rw_R^{-1}\ge \eta R^{-\alpha}
$$
for some $\eta>0$. Consequently, 
\be
w_R(t)\ge w_R(0)\exp\left(\eta R^{-\alpha} t\right).\label{wwww}
 \ee
 
\noindent
 Under  assumption \rf{(ii)}, there exist a constant $C>0$ and a sequence $R_n\to 0$ such that $w_{R_n}(t)\ge CR_n^{d-\alpha}\exp\left(\eta R_n^{-\alpha}t\right)$. 
 Thus, for any $\nu\in(0,\alpha)$, we obtain $w_{R_n}(t)>M$ for $t\ge T\sim R_n^{\alpha-\nu}$ asymptotically when $R_n\to 0$, which implies that  $u(r,t)$ cannot be defined on any interval $[0,T]$ with some $T>0$.   \qed

 \begin{remark}\label{r3} 
A sufficient condition \rf{www} for blowup can be expressed for $\alpha\ge 1$  in terms of the Morrey norm of $M^{d/\alpha}(\R^d)$, and we estimate that critical quantity sufficient for blowup asymptotically as $d\to\infty$. 
Observe that 
$$
c_{\alpha,d} =
\sigma_d\frac{\ell_\alpha}{C(\alpha,d)}\sim C_\alpha\sigma_d d^{\alpha/2-2}\ \ {\rm \ \ asymptotically \ as \ \ }d\to\infty$$
(by \rf{ell}, \rf{ww}) with 
\be
C_\alpha=2^{2+\alpha/2}\left(1+\frac{\alpha}{2}\right)^{1-\alpha/2}\Gamma\left(1+\frac{\alpha}{2}\right)^2\le 8.\label{const}
\ee 
The {\em radial concentration} of $u_0\ge 0$ appearing in the assumptions of Theorem \ref{blow}
\be
\xn u_0\xn \equiv \sup_{R>0}R^{\alpha-d}\int_{\{|y|<R\}} u_0(y)\dy, 
\label{conc}
\ee  
and the upper bound  of the moments  $\sup_{R>0}R^{\alpha-d}\int\psi_R(y)u_0(y)\dy$ are  equivalent. % to the Morrey norm of $u_0$ in $M^{d/\alpha}$. 
Indeed, since for each locally integrable function $\omega\ge 0$, each  $R>0$ and $s\in(0,1)$ we have   
\bea
\int \psi_R(x)\omega(x)\dx&\ge& \int_{\left\{|x|\le R\sqrt{1-s}\right\}}(1-(1-s))^{1+\alpha/2}\omega(x)\dx\nonumber\\
&=& s^{1+\alpha/2}\int_{\left\{|x|\le R\sqrt{1-s}\right\}}\omega(x)\dx,\nonumber
\eea 
and 
$$\max_{s\in[0,1]}s^{1+\alpha/2}(1-s)^{(d-\alpha)/2}=\left(\frac{\alpha+2}{d+2}\right)^{1+\alpha/2}\left(\frac{d-\alpha}{d+2}\right)^{(d-\alpha)/2}\equiv H_d.$$
Therefore $ \sup_{R>0} R^{\alpha-d}\int \psi_R\,\omega \ge H_d L$ if  $\xn \omega\xn>L$; {\em  i.e.} the upper bound of the moments,  the radial concentration $\xn \omega\xn$ of $\omega$ as well as  the $M^{d/\alpha}(\R^d)$ Morrey norm $\mn \omega\mnp$ for $\alpha\ge 1$ are comparable by Proposition \ref{prop:Mor}.  
 Note that  we have asymptotically 
 \be
 H_d^{-1}\sim  \left(\frac{d{\rm e}}{\alpha+2}\right)^{1+\alpha/2}.\label{H} 
 \ee 
Thus, condition \rf{www}  is satisfied if, {\em e.g.}, 
$$\xn u_0\xn> { C_\alpha}{d^{\alpha/2-2}}\sigma_d H_d^{-1}\sim {\widetilde C_\alpha}d^{\alpha-1}\sigma_d,$$ 
with  $C_\alpha$ as in \rf{const} and where 
\be 
{\widetilde C_\alpha}=2\left(1+\frac{\alpha}{2}\right)^{-\alpha}\Gamma\left(1+\frac{\alpha}{2}\right)^2{\rm e}^{1+\alpha/2}\le 2\left(1+\frac{\alpha}{2}\right)^{-\alpha}{\rm e}^{1+\alpha/2}\label{const-alpha}
\ee 
  --- and this leads to a blowup.  

 Therefore, we established   that  the asymptotic  discrepancy between the critical quantity for the global-in-time existence of solutions $\s_d\frac{s(\alpha,d)}{d-\alpha}$ in Theorem \ref{mainglobth} and the bound on the radial concentration %Morrey space $ M^{d/\alpha}(\R^d)$ norm 
 guaranteeing the finite time blowup,  is of order $d^{\alpha/2}$ because of relations \rf{H} and \rf{stala}, {\em i.e.}   $\xn u_C\xn=\frac{\s_d}{d-\alpha}s(\alpha,d) \sim 2^{\alpha/2}\frac{\Gamma(\alpha)}{\Gamma\left(\frac{\alpha}{2}\right)} \s_d d^{\alpha/2-1}$. 
\end{remark}
 
%\newpage  

 \end{document}